%% file: LiSun_Manuscript_1122.tex
\documentclass[11pt]{article}
\usepackage{etex}
\usepackage{amsmath,amssymb,amsfonts,amsthm}
\usepackage{a4wide}
\usepackage{color}
\usepackage{verbatim}
\usepackage{epsfig,subfigure,epstopdf,float}
\usepackage[]{graphics}
\usepackage{graphicx,inputenc}
\usepackage{subfigure}
\usepackage[justification=centering]{caption}
\usepackage{pictex}
\usepackage{wrapfig}
\usepackage{multirow}
\usepackage[T1]{fontenc}
\usepackage{authblk}
\usepackage[textwidth=6.0in,textheight=9in]{geometry}
\usepackage[linesnumbered,ruled,vlined, ruled,vlined]{algorithm2e}

\usepackage[colorlinks,linktocpage,linkcolor=blue]{hyperref}
\usepackage{fancyhdr}

\newtheoremstyle{exampstyle}
  {\topsep} 
  {\topsep} 
  {\itshape} 
  {} 
  {\bfseries} 
  {.} 
  {.5em} 
  {} 

\theoremstyle{exampstyle}
\numberwithin{equation}{section}
\newtheorem{theorem}{Theorem}
\newtheorem{lemma}{Lemma}[section]

\newtheorem{remark}[lemma]{Remark}

\newtheorem{corollary}[lemma]{Corollary}

\allowdisplaybreaks
\usepackage{parskip}
\let\oldref\ref
\renewcommand{\ref}[1]{(\oldref{#1})}  
\renewcommand{\eqref}[1]{(\oldref{#1})}

\newbox\boxaddrone \newbox\boxaddrtwo

\newcommand{\N}{\mathcal{N}}

\newcommand{\ba}{\begin{eqnarray*}}
\newcommand{\ea}{\end{eqnarray*}}
\newcommand{\be}{\begin{equation}}
\newcommand{\ee}{\end{equation}}
\newcommand{\bea}{\begin{eqnarray}}
\newcommand{\eea}{\end{eqnarray}}

\newbox\boxaddrone \newbox\boxaddrtwo

\def\N+{n\in\mathbb{N}^{+}}

\def\1d{\mathcal{D}((-\Delta)^{\gamma_1+1/2})}
\def\2d{\mathcal{D}((-\Delta)^{\gamma_2+1})}

\begin{document}
\title{Simultaneous uniqueness and numerical inversion for an inverse problem in the time-domain diffuse optical tomography with fluorescence}

\author[1]{Zhiyuan Li}
\author[2,3,4]{Chunlong Sun\thanks{Corresponding author: Dr. Chunlong Sun, email: sunchunlong@nuaa.edu.cn}}
\affil[1]{School of Mathematics and Statistics, Ningbo University, Ningbo 315211, Zhejiang, China}
\affil[2]{School of Mathematics, Nanjing University of Aeronautics and Astronautics, Nanjing 211106, Jiangsu, China}
\affil[3]{Key Laboratory of Mathematical Modelling and High Performance Computing of Air Vehicles (NUAA), MIIT, Nanjing, 211106, China}
\affil[4]{Nanjing Center for Applied Mathematics, Nanjing 211135, Jiangsu, China}

\maketitle

\begin{abstract}
\noindent In this work, an inverse problem on the determination of multiple coefficients arising from the time-domain diffuse optical tomography with fluorescence (DOT-FDOT) is investigated. We simultaneously recover the distribution of background absorption coefficient, photon diffusion coefficient as well as the fluorescence absorption in biological tissue by the time-dependent boundary measurements. We build the uniqueness theorem of this multiple coefficients simultaneous inverse problem. After that, the numerical inversions are considered. We introduce an accelerated Landweber iterative algorithm and give several numerical examples illustrating the performance of the proposed inversion schemes.\\

\\

\end{abstract}

\section{Introduction.}
\setcounter{equation}{0}
\subsection{Background and literature.} 
Optical tomography is rapidly gaining acceptance as an important diagnostic and monitoring tool of symptoms in medical applications \cite{NitziachristoV02, Mycek03, Rudin13}. Compared to X-ray computed tomography (CT), magnetic resonance imaging (MRI) and positron emission tomography (PET), optical tomography  possesses some advantages like low cost and portable, and therefore it has drawn more and more attention. 
Light propagation in highly scattering medium, such as biological tissue, is dominated by multiple scattering, and can be governed by the Boltzmann radiative transfer equation (RTE) \cite{Durduran10,Jiang11,Marttelli}.
However, the cost of solving RTE is extremely high, and therefore people approximately expand the radiant intensity as an isotropic {\it photon density} plus a small {\it photon flux}, and sequentially the transport equation can be reduced to a diffusion equation. The diffusion equation can be a sufficiently accurate approximation to the RTE, and has been applied for the quantitative analysis of optical properties of random media \cite{Arridge99,Arr09} due to its much lower computation than RTE. This kind of analysis includes diffuse optical spectroscopy (DOS) and diffuse optical tomography (DOT). Suppose the medium has the diffusion coefficient $D$ and absorption coefficient $\mu_a$. Both of DOS and DOT are to identify the unknown $(D,\mu_a)$ from the information of light measured on the boundary. 

Fluorescence diffuse optical tomography (FDOT) is one type of diffuse optical tomography that uses fluorescence light from fluorophores in biological tissue, for which two processes are coupled, namely, {\it excitation} and {\it emission} ({\it fluorescence}). The excitation photons injected on the boundary of the medium propagate to the fluorophores and then some photons are absorbed by them which excite the fluorophore molecules. After a moment of the absorption, the fluorophores emit other photons, fluorescence, at more longer wavelength than the wavelength of the excitation photons. Let $\Omega\subset\mathbb R^d$ be the background medium with its boundary $\partial\Omega$. We denote $\mu_f$ the distributions of fluorophores inside $\Omega$. FDOT is a method to achieve imaging of $\mu_f$ from the measured fluorescence on boundary. However, most existing works focus on the linear FDOT problem, which is predicated on knowing the background absorption $\mu_a$ and diffusion coefficient $D$, and the nonlinear dependence of excitation on $\mu_f$ is ignored. In this work, we study diffuse optical tomography with fluorescence (DOT-FDOT) in a more general framework where $(D,\mu_a)$ are unknown and the excitation also depends on $\mu_f$. It is very nature to consider this simultaneous inverse problem but really challenging.

Let $u_e(x,t)$ and $u_m(x,t)$ be the photon density of excitation light and emission light, respectively. As a forward problem for DOT-FDOT, we consider the following initial-boundary value problems for $u_e$ and $u_m$
\begin{equation}\label{excitation-Ue}
\begin{cases}
\begin{aligned}
c^{-1}\partial_t u_e -\text{div}(D(x)\nabla u_e) +(\mu_a(x)+\mu_f(x))u_e&=0,  &\quad& (x,t)\in\Omega\times (0,T),\\
u_e&=0,   &\quad&  (x,t)\in\Omega\times\{0\},\\
u_e&=g, &\quad& (x,t)\in\partial\Omega\times(0,T),
\end{aligned}
\end{cases}
\end{equation}
and
\begin{equation}\label{emission-um}
\begin{cases}
\begin{aligned}
c^{-1}\partial_t u_m-\text{div}(D(x)\nabla u_m) +\mu_a(x) u_m&=\mu_f(x) u_e, && (x,t)\in\Omega\times (0,T),\\
u_m&=0, && (x,t)\in\Omega\times \{0\},\\
u_m&=0, && (x,t)\in\partial\Omega\times(0,T).
\end{aligned}
\end{cases}
\end{equation}
Here 
$c>0$ is the speed of light in the medium; $D$ is the photon diffusion coefficient defined by $D:=1/(3\mu_s')$, where $\mu_s'>0$ is the reduced scattering coefficient. The DOT-FDOT in this work is formulated as follows.

{\bf Inverse problem:} Let $\Gamma_{\rm in}$ and $\Gamma$ be open sub-boundaries of $\partial\Omega$ satisfying 
\begin{equation}\label{obser-gamma}
\Gamma_{\rm in}\cup \Gamma = \partial\Omega,\quad \Gamma_{\rm in} \cap \Gamma\neq\emptyset.
\end{equation}
Suppose $D, \mu_a$ and $\mu_f$ are unknown in \eqref{excitation-Ue}-\eqref{emission-um}. By a suitable choice of the input $g$ supported on $\Gamma_{\rm in}\times[0,T]$, we consider to recover the $(D,\mu_a,\mu_f)$ simultaneously from the time-dependent measurements on the sub-boundary 
\begin{equation}\label{measure}
D\frac{\partial u_e}{\partial\nu}|_{\Gamma\times(0,T)}, \quad D\frac{\partial u_m}{\partial\nu}|_{\Gamma\times(0,T)},
\end{equation}
where $\nu=\nu(x)$ is the outward unit normal vector at $x\in\Gamma$. 

For these optical imaging problems, it is also possible to collect the observation data in the frequency domain at one or several modulation frequencies \cite{Arridge99}. Further, to overcome the ill-posedness of the inverse problems, the necessary regularization techniques such as Tikhonov regularization, sparse regularization methods and  hybrid regularization methods have been introduced  \cite{Correia10,Dutta12,Liu20}. Especially, the sparse regularization methods have additional advantages in promoting sparsity and higher spatial resolution for the cases that the target is relatively small compared to the background.  Above all, there already many works focus on improving the quality of the reconstruction and the reconstruction speed. However, in the aspect of theoretical analysis, especially the uniqueness of DOT-FDOT has not been rigorously investigated as far as we know. In this work, we will study the uniqueness of time-domain DOT-FDOT. The details are contained in the main theorem, which is stated in the next subsection.
\subsection{The main theorem.} 
In this work, the so-called time-domain DOT-FDOT based on the initial-boundary value problems \eqref{excitation-Ue}-\eqref{emission-um} is to solve the inverse problem:
\begin{equation}\label{IP-Total}
\text{recovering} \ \ (D,\mu_a,\mu_f) \ \ \text{from the measurements  \eqref{measure}}.
\end{equation}
The DOT-FDOT of \eqref{IP-Total} is equivalent to the two successive inverse problems given by
\begin{equation}\label{inv-new1}
\text{recovering} \ D \ \text{and} \ (\mu_a+\mu_f) \ \text{from the measurement}  \ D\frac{\partial u_e}{\partial\nu}|_{\Gamma\times(0,T)}
\end{equation}
and 
\begin{equation}\label{inv-new2}
\text{recovering} \ \mu_a \ and \ \mu_f \ \text{from the measurement}\  D\big(\frac{\partial u_e}{\partial\nu}+\frac{\partial u_m}{\partial\nu}\big)|_{\Gamma\times(0,T)}.
\end{equation}
To make our analysis more convenient, we set $c=1$ in models \eqref{excitation-Ue}-\eqref{emission-um}. Now it is time to state the uniqueness theorem. 

\begin{theorem}\label{uniqueness}
Let $\Gamma_{\rm in}$ and $\Gamma$ be nonempty open sub-boundaries of $\partial\Omega$, respectively, and fulfill \eqref{obser-gamma}. For the two sets of unknowns $\big\{D,\mu_a, \mu_f \big\}$ and $\big\{\tilde D,\tilde \mu_a, \tilde \mu_f \big\}$, by $\{u_e,u_m\}$ and $\{\tilde u_e, \tilde u_m\}$ we respectively denote the corresponding solutions of initial-boundary value problems \eqref{excitation-Ue}-\eqref{emission-um} with boundary input $g$ satisfying \eqref{g-input}. Then we conclude that:
\begin{itemize}
\item [(\romannumeral 1)] $D\frac{\partial u_e}{\partial\nu}|_{\Gamma\times(0,T)}=\tilde D\frac{\partial \tilde u_e}{\partial\nu}|_{\Gamma\times(0,T)}$ with a priori $\nabla D = \nabla \tilde D$ on $\partial\Omega$ implies
\begin{equation}\label{uniqueness-IP1}
D =\tilde D,\quad \mu_a+\mu_f=\tilde \mu_a+\tilde \mu_f;
\end{equation}
\item [(\romannumeral 2)] with result \eqref{uniqueness-IP1} and $D\frac{\partial u_e}{\partial\nu}|_{\Gamma\times(0,T)}=\tilde D\frac{\partial \tilde u_e}{\partial\nu}|_{\Gamma\times(0,T)}$, $D\frac{\partial u_m}{\partial\nu}|_{\Gamma\times(0,T)}=\tilde D\frac{\partial \tilde u_m}{\partial\nu}|_{\Gamma\times(0,T)}$, we have 
\begin{equation}\label{uniqueness-IP2}
\mu_a=\tilde \mu_a, \quad \mu_f=\tilde \mu_f.
\end{equation}
\end{itemize}
\end{theorem}
\begin{remark}\label{Rem1-DOT-FDOT}
If $\mu_f=0$ in the initial boundary value problem \eqref{excitation-Ue}, the reconstructions of $\mu_a$ and $D$ from the sub-boundary measurement $D\frac{\partial u_e}{\partial\nu}|_{\Gamma\times(0,T)}$ are known as the DOT. Given $(D,\mu_a)$ and suppose $\mu_f=0$ in \eqref{excitation-Ue}, then the inverse problem of recovering $\mu_f$ from  the data $D\frac{\partial u_m}{\partial\nu}|_{\Gamma\times(0,T)}$ is the linear FDOT, which is the common problem in the community of fluorescence tomography. Hence, Theorem \ref{uniqueness} covers the uniqueness result of these two common optical imaging problems but needs the both boundary flux of excitation $u_e$ and fluorescence $u_m$. 
\end{remark}
\subsection{Contribution and outline.}
The literature on the uniqueness of time-domain DOT-FDOT is relatively rare and the difficulty to prove the uniqueness of such inverse problem is addressed in the overview \cite{Arridge99}, especially the identifiability of multiple coefficients from the boundary measurement. Some aspects of uniqueness and ill-posedness of DOT are considered \cite{Anikonov84,Anikonov85,Arridge98,Bal09,canuto2001,Choulli,dolle2018,Harrach09,He2000}, and the uniqueness results are related with the applied DOT or FDOT models. It is in general not possible to simultaneously recover the diffusion and absorption coefficients in steady-state (dc) DOT \cite{Arridge98}. Still for steady-state DOT, if restricting the unknown parameters to piecewise constant diffusions and piecewise analytic absorption coefficients, both parameters can simultaneously be determined from complete measurement data on an arbitrarily small part of the boundary \cite{Harrach09}. Furthermore, although DOT is a nonlinear inverse problem, a linearized version is often considered. The uniqueness of linearized time-domain DOT is considered \cite{He2000}. For the nonliner time-domain DOT, the uniqueness can be regained but need the Dirichlet-to-Neumann (DtN) mapping data on the sub-boundary \cite{canuto2001,dolle2018}. For linear time-domain FDOT, an identifiability result of  $\mu_f$ is established in \cite{Liu:2020}, but it requires strong prior assumptions on $\mu_f$. Recently, the uniqueness of identifying unknown $\mu_f$ from arbitrary sub-boundary data is established \cite{Sun2022} by assuming the semi-discrete form of right-hand source. However, we see that the known uniqueness results require either the linearized simple model or the impractical DtN mapping measurement. 

In this work, the uniqueness of time-domain DOT-FDOT in a more general framework is established. The inverse problem under consideration is more complicated and challenged in the following sense
\begin{itemize}
\item [(\romannumeral 1)] The initial-boundary value problems \eqref{excitation-Ue}-\eqref{emission-um} are coupled and the excitation $u_e$ depends on $\mu_f$ nonlinearly. Therefore the concerned inverse problems \eqref{inv-new1} and \eqref{inv-new2} are nonlinear. 
\item [(\romannumeral 2)] The problem of recovering the parameters $D,\mu_a$ and $\mu_f$ is a "mixed" inverse problem. The used data is one single measurement on the sub-boundary $\Gamma\times(0,T)$, not the DtN mapping data.
\end{itemize}

Theorem \ref{uniqueness} confirms that the unknowns $(D,\mu_a,\mu_f)$ can be uniquely determined by the arbitrary sub-boundary measurements $D\frac{\partial u_e}{\partial\nu}|_{\Gamma\times(0,T)}$ and $D\frac{\partial u_m}{\partial\nu}|_{\Gamma\times(0,T)}$. The observation area $\Gamma$ can be an arbitrarily open subset of boundary $\partial\Omega$. Moreover, Theorem \ref{uniqueness} covers the uniqueness results of the common DOT problem and FDOT problem as mentioned in Remark \ref{Rem1-DOT-FDOT}. 

The rest of this article is organized as follows. In Section \ref{sec_pre}, we collect several preliminary works. We prove the uniqueness theorem in Section \ref{sec_uni}. The numerical inversions will be considered in Section \ref{sec_num}. We show the validity of the proposed scheme by several numerical examples. Finally, Section \ref{concl} is devoted to conclusion and future works.

\input{pre_proof}

\input{num}

\section{Conclusion and future works.}\label{concl}
This work studies a multiple coefficients inversion problem arising from the time-domain diffuse optical tomography with fluorescence, which is to recover the diffusion coefficient, background absorption as well as fluorescence absorption from one single boundary measurement. We deduce the uniqueness theorem and do the numerical inversion by the Nesterov accelerated Landweber iteration. The numerical results are satisfactory in view of the high ill-posedness of our simultaneous inverse problem. However, to make the theory and the numerical experiments in this work more applicable, we will focus on the following two aspects of future works.

Firstly,  in this work, we assumed that the absorption and scattering coefficients for the excitation and emission lights are the same, which leads to the same $D$ and $\mu_a$ in \eqref{excitation-Ue} and \eqref{emission-um}. 
However, in many practical applications,  the absorption and scattering coefficients of the background medium are not same for the lights with different wave lengths. Hence one of our next works is to prove the uniqueness theorem of recovering $(D_e,D_m,\mu_a^e,\mu_a^m,\mu_f)$ for the DOT-FDOT with the governing equations given by
$$
c^{-1}\partial_t u_e -\text{div}(D_e\nabla u_e) +(\mu_a^e+\mu_f)u_e=0,
$$
and
$$
c^{-1}\partial_t u_m -\text{div}(D_m\nabla u_e) +\mu_a^m u_m=\mu_f u_e.
$$
Secondly, to save the computational costs, we only consider the simultaneous inversion of background absorption and fluoresence absorption in the numerical experiments in this work. The fast algorithms of recovering the background absorption, fluoresence absorption as well as the diffusion coefficient are important and worth our further investigations.
\section*{Acknowledgment.} 
Zhiyuan Li thanks National Natural Science Foundation of China (Grant No.12271277). Chunlong Sun thanks National Natural Science Foundation of China (Grant No.12201298), Natural Science Foundation of Jiangsu Province, China (Grant No.BK20210269) and “Double Innovation” Doctor of Jiangsu Province, China (Grant No.JSSCBS20220227). This work is partially supported by the Open Research Fund of Key Laboratory of Nonlinear Analysis \& Applications (Central China Normal University), Ministry of Education, P. R. China.

\bibliographystyle{plainurl} 
\bibliography{ref}
\end{document}

%% file: pre_proof.tex
\section{Some preliminary results.}\label{sec_pre}
In this section, we first fix some general settings and notations, and show several lemmas which will be used in the forthcoming discussion.
\subsection{Settings.}
Now we give several settings. We first define the operator $A_e$ by
$$
A_eu:=  - \text{div}\left(D(x) \nabla u \right) + (\mu_a(x) + \mu_f(x))u,\quad u\in D(A):=H^2(\Omega)\cap H_0^1(\Omega).
$$
Since $D(x)>0$ and $\mu_e+\mu_f\ge0$,  the operator $A_e$ is symmetric uniformly elliptic. Let $\{\lambda_n,\varphi_n\}_{n=1}^\infty$ be the Dirichlet eigensystem of the operator $A_e$, where $0<\lambda_1\le \lambda_2\le \cdots $ and the corresponding eigenfunctions $\{\varphi_n\}_{n=1}^\infty$ forms an orthonormal basis of $L^2(\Omega)$. 
 
If considering the algebraic multiplicity, we need to make slight changes to the above symbols. Indeed, by $\{\lambda_k\}_{k\in\mathbb N}$ and $m_k\in\mathbb N$ we denote the strictly increasing sequence of the eigenvalues of $A_e$ and the algebraic multiplicity of $\lambda_k$, respectively. For each eigenvalue $\lambda_k$, we introduce a family $\{\varphi_{k,\ell}\}_{\ell=1}^{m_k}$ of eigenfunctions of $A_e$, i.e.,
\[
A_e\varphi_{k,\ell}=\lambda_k\varphi_{k,\ell},\quad\ell=1,\ldots,m_k,
\]
which forms an orthonormal basis in $L^2(\Omega)$ of the algebraic eigenspace of the operator $A_e$ associated with $\lambda_k$.  
 
Next we specify the choice of the Dirichlet boundary input $g(x,t)$, which plays an essential role in the consideration of our inverse problem. Following the settings in the paper from Kian et al. \cite{Kian20}, we choose a sequence of real and nonnegative functions  $\{\psi_k\}_{k\in\mathbb N}$ define on $\overline{\mathbb R_+}:=[0,\infty)$ such that $0\le\psi_k\le1$ and
\begin{equation}\label{psi_k}
 C^\infty(\overline{\mathbb R_+}) \ni \psi_k=\begin{cases}
0 & \mbox{on }[0,t_{2k-2}],\\
1 & \mbox{on }[t_{2k-1},\infty),
\end{cases}\end{equation}
where $\{t_k\}_{k=0}^\infty$ is a strictly increasing sequence satisfying $t_0=0$ and $\lim_{k\to\infty}t_k=T$. We fix also $\{c_k\}_{k=0}^\infty$ a sequence of $\mathbb R_+:=(0,\infty)$ such that
\begin{equation}\label{condi-psi}
\sum_{k=1}^\infty c_k\|\psi_k\|_{W^{2,\infty}(\mathbb R_+)}<\infty.
\end{equation}
For example, one can choose $c_k:=e^{-k^2}$. We select a sequence of functions $\{\eta_k\}_{k\in\mathbb N}$ of $H^{\frac32}(\partial\Omega)$ such that $\|\eta_k\|_{H^{\frac32}(\partial\Omega)} \le 1$ and span$\{\eta_k\}$ is dense in $H^{\frac32}(\partial\Omega)$. Finally, we can construct the Dirichlet boundary input $g$ of the system \eqref{excitation-Ue}
\begin{equation}\label{g-input}
g(x,t):=\chi(x)\sum_{k=1}^\infty c_k\,\psi_k(t)\eta_k(x),\quad (x,t)\in \partial\Omega\times \mathbb R_+,
\end{equation}
where $0\le\chi\le1$ is such that 
\begin{equation}\label{def-chi}
C_0^\infty(\partial\Omega)\ni\chi=\begin{cases}
1 &\mbox{ on } \Gamma_{\rm in}',\\
0 &\mbox{ on } \partial\Omega\setminus\Gamma_{\rm in},
\end{cases}
\end{equation}
where $\Gamma_{\rm in}'\subset \Gamma_{\rm in}\subset\partial\Omega$. Based on the above settings of boundary input $g$, it is not difficult to check that the function $g$ belongs to the function space $C^2(\overline{\mathbb R_+};H^{\frac32}(\partial\Omega))$. 
\subsection{Time-analyticity of solutions.}\label{sec-analytic}
Now based on the characteristics of boundary condition $g$ in \eqref{g-input}, we construct the following subproblems corresponding to the system \eqref{excitation-Ue} by treating each term of the series as a new boundary input, that is,
\begin{equation}\label{Ue-sub}
\left\{
\begin{alignedat}{2}
\partial_t u_{e,k} + A_e u_{e,k}&=0 &\quad  &\mbox{in } \Omega\times \mathbb R_+,\\
u_{e,k}&=0 &\quad  &\mbox{in }\Omega\times\{0\},\\
u_{e,k}& = \chi c_k\,\psi_k\,\eta_k &\quad  &\mbox{on }\partial\Omega \times \mathbb R_+.
\end{alignedat}
\right.
\end{equation}
We investigate the analyticity of the solutions $u_{e,k}$ to \eqref{Ue-sub} in time for each $k\in\mathbb N$. To this end, by $\mathcal H(r,\infty;X)$ we denote the set of holomophic functions on $(r,+\infty)$ taking values in the Banach space $X$. Then we state the following intermediate result.
\begin{lemma}\label{l1}
For any integer $k\in\mathbb N$, the solution $u_{e,k}$ to the sub-problem \eqref{Ue-sub} belongs to $L^\infty(\mathbb R_+;H^2(\Omega))$ and is such that $u_{e,k}|_{(t_{2k-1},\infty)}\in\mathcal H(t_{2k-1},\infty;H^2(\Omega))$.
\end{lemma}
\begin{proof}
Fix $G_k$ solving the boundary value problem
$$
\left\{ \begin{array}{rcll} 
A_e G_k & = & 0 & \mbox{in } \Omega ,\\
  G_k & = & \chi\eta_k & \mbox{on}\ \partial\Omega.
\end{array}
\right.
$$
It is not difficult to check from the theory of elliptic equations that $G_k\in H^2(\Omega)$ and $G_k$ is such that the following regularity estimate
\begin{equation}\label{esti-G}
\|G_k\|_{H^2(\Omega)} \le C\|\eta_k\|_{H^{\frac32}(\partial\Omega)} \le C.
\end{equation}
Here and henceforth, by $C>0$ we denote a general constant which is independent of $k$ and  may change by lines. Then, we can decompose $u_{e,k}$ into $u_{e,k} = w_{e,k} + v_{e,k}$ with $v_{e,k}(t,\cdot) = c_k\psi_k(t)G_k$ and $w_{e,k}$ solving
\begin{equation*}
\left\{ 
\begin{alignedat}{2} 
(\partial_t + A_e) w_{e,k}  & =  -c_k \psi_k'(t)G_k &\quad & \mbox{in } \Omega\times\mathbb R_+,\\
w_{e,k}(x,t)  & =  0 &\quad & \mbox{on } \partial\Omega\times \mathbb R_+, \\  
w_{e,k}(\cdot,0) & = 0 &\quad & \mbox{in } \Omega.
\end{alignedat}
\right.
\end{equation*}
It is clear that $v_{e,k}\in \mathcal H(t_{2k-1},\infty;H^2(\Omega))$ by the notation of \eqref{psi_k}. Therefore, we only need to prove that $w_{e,k}\in\mathcal H(t_{2k-1},\infty;H^2(\Omega))$. For this, by noticing the argument of eigenfunction expansion, we can arrive at the representation formula of the solution $w_{e,k}$ as follows.
\begin{equation}\label{sol-w}
w_{e,k}(\cdot,t) 
= -c_k\sum_{\ell=1}^\infty \int_0^t \exp(-\lambda_\ell (t-s))\psi_k'(s)ds\left\langle G_k,\varphi_\ell\right\rangle_{L^2(\Omega)} \varphi_\ell.
\end{equation}
Using the fact that $\psi_k'=0$ on $(t_{2k-1},+\infty)$, it follows that
$$
w_{e,k}(\cdot,t) = -c_k\sum_{\ell=1}^\infty e^{-\lambda_\ell t} \int_0^{t_{2k-1}}e^{\lambda_\ell s}\psi_k'(s)ds\left\langle G_k,\varphi_\ell\right\rangle_{L^2(\Omega)} \varphi_\ell,\quad t\geq t_{2k-1},
$$
from which by further noting the analyticity of the exponential function we conclude that $w_{e,k}|_{(t_{2k-1},\infty)}\in\mathcal H(t_{2k-1},+\infty;H^2(\Omega))$. Therefore, $u_{e,k}$ restricted to $(t_{2k-1},\infty)$ belongs to $\mathcal H(t_{2k-1},\infty;H^2(\Omega))$. 

We next show that the solution $u_{e,k}\in L^\infty(\mathbb R_+;H^2(\Omega))$. It is not difficult to check that $v_{e,k} \in L^\infty(\mathbb R_+;H^2(\Omega))$, so it remains to deal with $w_{e,k}$. For this, on the basis of the representation formula \eqref{sol-w} and noting $\|u\|_{H^2(\Omega)} \le C\|A_e u\|_{L^2(\Omega)}$, by a directly calculation we arrive at the inequalities
\begin{align*}
&\|w_{e,k}(\cdot,t)\|_{H^2(\Omega)}^2 
 \le C\|A_e w_{e,k}(\cdot,t)\|_{L^2(\Omega)}^2\\
& \le C|c_k|^2\sum_{\ell=1}^\infty \lambda_\ell^2 \left|\int_0^t \exp\left\{-\lambda_\ell(t-s)\right\} \psi_k'(s) d s\right|^2\left|\langle G_k,\varphi_\ell\rangle_{L^2(\Omega)}\right|^2\\
& \le C|c_k|^2 \|\psi_k\|_{W^{1,\infty}(\mathbb R_+)}^2 \sum_{\ell=1}^\infty \lambda_\ell^2 \left|\int_0^t e^{-\lambda_\ell s} d s\right|^2\left|\langle G_k,\varphi_\ell\rangle_{L^2(\Omega)}\right|^2\\
& = C|c_k|^2 \|\psi_k\|_{W^{1,\infty}(\mathbb R_+)}^2 \sum_{\ell=1}^\infty  \left| 1-e^{-\lambda_\ell t} \right|^2\left|\langle G_k,\varphi_\ell\rangle_{L^2(\Omega)}\right|^2
 \le C \|G_k\|_{L^2(\Omega)}^2 \le C.
\end{align*}
Here the last inequality is due to the assumption \eqref{condi-psi} and the regularity estimate \eqref{esti-G}. Therefore we have the function $
t\longmapsto w_{e,k}(\cdot,t)$ is in $L^\infty(\mathbb R_+;H^2(\Omega))$. 
Collecting all the above estimates, we can finish the proof of the lemma.
\end{proof}
\section{The proof of the uniqueness theorem.}\label{sec_uni} 
As above mentioned, the unique solvability of our DOT-FDOT problem \eqref{IP-Total} is equivalent to the uniqueness of the two sub inverse problems: \eqref{inv-new1} and \eqref{inv-new2}. Consequently, we first treat the sub-problem \eqref{inv-new1}, that is, the determination of the diffusion coefficient $D$ and mixture absorption $\mu_a+\mu_f$ in \eqref{excitation-Ue}. Then we proceed to the uniqueness of the reconstruction of background absorption $\mu_a$ and fluorescence absorption $\mu_f$ (i.e, the sub-problem \eqref{inv-new2}) by employing the similar argument used in solving \eqref{inv-new1}. 
\subsection{Recovering the diffusion coefficient and mixture absorption.}\label{subsec-uniqu1}
In this part, we consider the determination of the diffusion coefficient $D$ and the mixture absorption $\mu_a + \mu_f$ from the boundary measurement 
\begin{equation}
D(x)\frac{\partial u_e}{\partial\nu}(x,t),\quad (x,t)\in \Gamma\times(0,T).
\end{equation}
Firstly, we transfer the observation information
\begin{equation}\label{obser}
D\frac{\partial u_e}{\partial\nu}|_{\Gamma\times(0,T)} = \tilde D\frac{\partial \tilde u_e}{\partial\nu}|_{\Gamma\times(0,T)}
\end{equation}
of original problem \eqref{excitation-Ue} to a sequential identities of Neumann boundary of the solutions to the corresponding subproblems \eqref{Ue-sub}, that is, we have the following useful lemma.
\begin{lemma}\label{lem-sub}
The assumption \eqref{obser} implies that 
\begin{equation}\label{obser-sub}
D\frac{\partial u_{e,k}}{\partial \nu} = \tilde D \frac{\partial \tilde u_{e,k}}{\partial \nu}\quad\mbox{on }\Gamma\times(0,\infty),\ k\in\mathbb N,
\end{equation}
where $u_{e,k}$ and $\tilde u_{e,k}$ are solutions of \eqref{Ue-sub} with respect to $(D,\mu_a + \mu_f)$ and $(\tilde D, \tilde \mu_a + \tilde \mu_f)$. 
\end{lemma}
\begin{proof}
We will prove the conclusion \eqref{obser-sub} by applying inductive argument. Firstly, for $k=1$, from the choice of the sequence $\{\psi_k\}_{k\in\mathbb N}$ in \eqref{psi_k}, it is not difficult to see that
$$
\psi_k=0\quad\mbox{in }(0,t_2),\quad \forall\,k\ge2.
$$
Therefore, in the case of $t\in (0,t_2)$, the boundary condition $g$ can be deduce into 
$$
g(x,t)=c_1\psi_1(t)\chi(x)\eta_1(x),\quad  (x,t)\in \partial\Omega\times(0,t_2).$$
Now we conclude from the wellposeness of the problem \eqref{excitation-Ue} that $u=u_{e,1}$ and $\tilde u = \tilde u_{e,1}$ in $\Omega\times(0,t_2)$, which further combined with the condition \eqref{obser} implies
\begin{equation}\label{t1f}
D \frac{\partial u_{e,1}}{\partial \nu} = \tilde D\frac{\partial \tilde u_{e,1}}{\partial \nu} \quad\mbox{on } \Gamma\times(0,t_2).
\end{equation}
On the other hand, from Lemma \ref{l1} we know that $u_{e,1},\tilde u_{e,1}\in \mathcal H(t_1,\infty;H^2(\Omega))$. Thus, the trace theorem and properties of the analytic functions yield that the following two functions 
$$
t\longmapsto D\frac{\partial u_{e,1}}{\partial \nu}(\cdot,t)|_{\Gamma} \mbox{ and } 
t\longmapsto\tilde D \frac{\partial \tilde u_{e,1}}{\partial \nu}(\cdot,t)|_{\Gamma},\quad t>t_1
$$
belong to $\mathcal H(t_1,\infty; L^2(\Gamma))$, and the condition \eqref{t1f} implies \eqref{obser-sub} for $k=1$. We finish the proof of \eqref{obser-sub} with $k=1$. 
Next assuming that \eqref{obser-sub} is fulfilled for all $k=1,2,\ldots,\ell$ with some $\ell\in\mathbb N$, we will show that the conclusion \eqref{obser-sub} is valid also for $k=\ell+1$. For this, again noting the definition of $\psi_k$ in \eqref{psi_k}, it can be easily seen that
$$
\psi_k=0\quad\mbox{in }(0,t_{2\ell+2}),\quad \forall\,k\ge\ell+2.
$$
Therefore, for $t\in(0,t_{2\ell+2})$, the Dirichlet boundary input $g$ can be rewritten as follows 
$$
g(x,t) = \chi(x)\sum_{k=1}^{\ell+1} c_k\psi_k(t)\eta_k(x),\quad  (x,t)\in \partial\Omega\times(0,t_{2\ell+2}).
$$
By an argument similar to the case of $k=1$, we can obtain
$$
\sum_{k=1}^{\ell+1} u_{e,k} = u_e,\quad \sum_{k=1}^{\ell+1} \tilde u_{e,k} = \tilde u_e\quad\mbox{in }\Omega\times(0,t_{2\ell+2}).
$$
Therefore, \eqref{obser} implies
$$
D\sum_{k=1}^{\ell+1}  \frac{\partial u_{e,k}}{\partial \nu} = \tilde D\sum_{k=1}^{\ell+1}  \frac{\partial \tilde u_{e,k}}{\partial\nu} \quad\mbox{on }\Gamma\times(0,t_{2\ell+2}).
$$
Now by the inductive assumption, that is, \eqref{obser-sub} holds for $k=1,\cdots,\ell$, we further obtain that
\[
D \frac{\partial u_{e,\ell+1}}{\partial \nu} = \tilde D \frac{\partial \tilde u_{e,\ell+1}}{\partial \nu} \quad \mbox{on }\Gamma \times (0,t_{2\ell+2}).
\]
Again by the use of the $t$-analyticity result in Lemma \ref{l1}, we conclude that the following two functions are such that
$$
t\longmapsto D\frac{\partial u_{e,\ell+1}}{\partial \nu}(\cdot,t)|_{\Gamma} \mbox{ and }
t\longmapsto \tilde D \frac{\partial \tilde u_{e,\ell+1}}{\partial \nu}(\cdot,t)|_{\Gamma} \in \mathcal H(t_{2\ell+1},\infty;L^2(\Gamma)),
$$
from which we further use the continuation principle of analytic functions to imply \eqref{obser-sub} holds true for $k=\ell+1$. This proves that \eqref{obser-sub} is valid for all $k\in\mathbb N$. We complete the proof of the lemma.
\end{proof}
\begin{remark}
Here and henceforth, some parts of the proofs may be similar to the argument used in Kian et al. \cite{Kian20}. But we retain this proof for the completeness and consistency of the article because our paper deals with a totally different model.
\end{remark}
From the results in Lemma \ref{l1} and Sobolev trace theorem, we can directly deduce that
\begin{corollary}
For any $k\in \mathbb N$ and any $\xi>0$, we have
$$
t\longmapsto e^{-\xi t}u_{e,k}(\cdot,t)\in L^1(\mathbb R_+;H^2(\Omega))
$$
and
$$
t\longmapsto e^{-\xi t} \frac{\partial u_{e,k}}{\partial\nu}(\cdot,t)\in L^1(\mathbb R_+;L^2(\Gamma)).
$$
\end{corollary}
The above statement ensures the existence of the Laplace transforms of the two functions $u_{e,k}$ and $\tilde u_{e,k}$, which allows us to use Laplace transform to consider our inverse problem in the frequency domain. Firstly, for any $\xi>0$ we introduce the following boundary value problem for elliptic equation
\begin{equation}\label{t1h}
\left\{
\begin{alignedat}{2}
(A_e + \xi) U_\xi(x)&=0 &\quad & \mbox{in }\Omega,\\
U_\xi(x)&= \chi(x)h(x)  &\quad & \mbox{on }\partial\Omega,
\end{alignedat}
\right.
\end{equation}
where $h\in H^{\frac32}(\partial\Omega)$ and the cut-off function $\chi$ is defined in \eqref{def-chi}. We first assert that the function $U_\xi$ is analytic with respect to the parameter $\xi>0$ and can be analytically extended to the set $\mathbb C\setminus \{-\lambda_k\}_{k=1}^\infty$. We have the following lemma.
\begin{lemma}\label{lem-U-analy}
The function $\mathbb R_+ \ni\xi \mapsto U_\xi\in H^2(\Omega)$ is analytic with respect to $\xi>0$ and can be analytically extended to $\xi\in \mathbb C\setminus \{-\lambda_k\}_{k=1}^\infty$.
\end{lemma}
\begin{proof}
Since $h\in H^{\frac32}(\partial\Omega)$, noting the Sobolev trace theorem, we can find a function $H\in H^2(\Omega)$ such that $H|_{\partial\Omega} = -\chi h$ in the trace sense and satisfies the estimate
\begin{equation}\label{esti-H}
\|H\|_{H^2(\Omega)} \le C\|h\|_{H^{\frac32}(\partial\Omega)}.
\end{equation}
Letting $W_\xi:= U_\xi + H$, it is not difficult to see that the function $W_\xi$ reads the following homogeneous boundary value problem
$$
\left\{ \begin{array}{rcll} 
(A_e + \xi) W_\xi & = & (A_e + \xi)H & \mbox{in } \Omega ,\\
  W_\xi & = & 0 & \mbox{on}\ \partial\Omega.
\end{array}
\right.
$$
Now from the argument of the Fourier series expansion, we assert that the function $W_\xi$ can be represented by
\begin{equation}\label{sol-W_xi}
W_\xi = \sum_{k=1}^\infty \sum_{\ell=1}^{m_k} \langle W_\xi,\varphi_{k,\ell} \rangle_{L^2(\Omega)} \varphi_{k,\ell}
= \sum_{k=1}^\infty \frac{1}{\xi + \lambda_k} \sum_{\ell=1}^{m_k} \big\langle (A_e + \xi)H,\varphi_{k,\ell} \big\rangle_{L^2(\Omega)} \varphi_{k,\ell},
\end{equation}
which combined with the relation $W_\xi=U_\xi + H$ gives the representation formula of the solution $U_\xi$:
\begin{equation}\label{sol-U_xi}
U_\xi =\sum_{k=1}^\infty \sum_{\ell=1}^{m_k} \langle U_\xi,\varphi_{k,\ell} \rangle_{L^2(\Omega)} \varphi_{k,\ell} = \sum_{k=1}^\infty \frac{1}{\xi + \lambda_k} \sum_{\ell=1}^{m_k} \big\langle (A_e + \xi)H,\varphi_{k,\ell} \big\rangle_{L^2(\Omega)} \varphi_{k,\ell} - H.
\end{equation}
We next show that the function $\mathbb R_+\ni\xi \mapsto W_\xi\in H^2(\Omega)$) can be analytically extended to the complex domain $\xi\in \mathbb C\setminus \{-\lambda_k\}_{k=1}^\infty$. For this, fixing any compact subset $K$ of $\mathbb C\setminus \{-\lambda_k\}_{k=1}^\infty$, on the basis of the representation formula \eqref{sol-U_xi} and the estimate $\|u\|_{H^2(\Omega)} \le C\|A_e u\|_{L^2(\Omega)}$, by a directly calculation we arrive at the inequalities
\begin{align*}
\|W_\xi(\cdot\,)\|_{H^2(\Omega)}^2 
& \le C\|A_e W_\xi(\cdot\,)\|_{L^2(\Omega)}^2\\
& \le C\sum_{k=1}^\infty  \frac{\lambda_k^2}{(\xi + \lambda_k)^2} \left| \sum_{\ell=1}^{m_k} \langle (A_e + \xi)H,\varphi_{k,\ell} \rangle_{L^2(\Omega)}\right|^2.
\end{align*}
Since $K\subset\subset \mathbb C\setminus \{-\lambda_k\}_{k=1}^\infty$, we see that $\frac{\lambda_k^2}{(\xi + \lambda_k)^2} \le C_K$, where the constant $C_K$ denotes a general positive constant which is independent of $k$ and may change by lines, which further implies that
\begin{align*}
\|W_\xi(\cdot\,)\|_{H^2(\Omega)}^2 
\le  C_K \|(A_e+\xi)H\|_{L^2(\Omega)}^2 \le C_K\|h\|_{H^{\frac32}(\partial\Omega)}.
\end{align*}
Here the last inequality is due to the estimate \eqref{esti-H}. Consequently, we assert that the series in \eqref{sol-W_xi} is uniformly convergent for any $\xi\in K$. Moreover, noting that the compact subset $K$ can be arbitrarily chosen in $\mathbb C\setminus \{-\lambda_k\}_{k=1}^\infty$, we have the required conclusion of the function $W_\xi$. Finally, since $H$ does not depend on $\xi$, we obtain that the map $\mathbb R_+ \ni\xi\mapsto U_\xi\in H^2(\Omega)$ can be analytical extended to the set $\mathbb C\setminus \{-\lambda_k\}_{k=1}^\infty$. We finish the proof of the lemma.
\end{proof}
We associate the problem \eqref{t1h} with the DtN mapping $\Lambda_\xi: H^{\frac32}(\partial\Omega) \to L^2(\Gamma)$ which is defined by
$$
\Lambda_\xi:h\longmapsto D \frac{\partial U_\xi}{\partial \nu} |_{\Gamma},\quad \xi>0.
$$
Similarly, we can define the DtN map 
$$
\tilde \Lambda_\xi: h\in H^{\frac32}(\partial\Omega)\longmapsto \tilde D \frac{\partial \tilde U_\xi}{\partial \nu} |_{\Gamma}
$$
corresponding to the elliptic operator $\tilde A_e$:
$$
\tilde A_e:=-\text{div} (\tilde D\nabla) + \tilde\mu_a + \tilde\mu_f.
$$
On the basis of the results in Lemma \ref{obser-sub}, we can obtain the following lemma.
\begin{lemma}\label{lem-Lambda}
The condition \eqref{obser-sub} implies $\Lambda_\xi = \tilde \Lambda_\xi$ for any $\xi>0$.
\end{lemma}
\begin{proof}
Applying the Laplace transform $\mathcal L$ in time on both sides of \eqref{obser-sub} implies
$$
\Lambda_\xi (c_k\eta_k\mathcal L\psi_k) = \tilde \Lambda_\xi (c_k\eta_k\mathcal L\psi_k),\quad\xi>0,\ k\in\mathbb N.
$$
Moreover, in view of the assumption $\psi_k\ge0,\not\equiv0$, we see that the Laplace transform $(\mathcal L\psi_k)(\xi)>0$ is valid for all $\xi>0$, which combined with the fact $c_k>0$ and the linearity of $\Lambda_\xi$ and $\tilde \Lambda_\xi$ implies
\[
\Lambda_\xi h = \tilde \Lambda_\xi h,\quad \forall \xi>0,\ \forall\,h\in {\rm span}\{\eta_k\}_{k\in\mathbb N}.
\]
Finally, $\Lambda_\xi = \tilde \Lambda_\xi$ is valid for any $\xi>0$ due to the density of ${\rm span}\{\eta_k\}$ in $H^{\frac32}(\partial\Omega)$. We finish the proof of the lemma.
\end{proof}
Recalling the notations of the Dirichlet eigensystem $\{\lambda_k, \{\varphi_{k,\ell}\}_{\ell=1}^{m_k}\}_{k=1}^\infty$ of the elliptic operator $A_e$, we define the sequence of functions:
\[
\Theta_k(x,x'):=D(x)D(x')\sum_{\ell=1}^{m_k} \frac{\partial \varphi_{k,\ell}}{\partial \nu}(x) \frac{\partial \varphi_{k,\ell}}{\partial\nu}(x'),\quad x, x'\in\partial\Omega,\ k\in\mathbb N.
\]
\begin{lemma}\label{lem-Theta}
The observation data \eqref{obser} implies $ \lambda_k = \tilde \lambda_k$ and 
\begin{equation}\label{eq-theta}
\Theta_k(x,x') = \tilde \Theta_k(x,x'), \quad  x\in\Gamma,\ x'\in\Gamma_{\rm in}',
\end{equation}
where $k\in \mathbb N$ and $\{\lambda_k, \{\varphi_{k,\ell}\}_{\ell=1}^{m_k}\}_{k=1}^\infty$ and $\{\tilde \lambda_k, \{\tilde \varphi_{k,\ell}\}_{\ell=1}^{\tilde m_k}\}_{k=1}^\infty$ are the Dirichlet eigensystems of the elliptic operators $A_e$ and $\tilde A_e$.
\end{lemma}
\begin{proof}
In the first step, we will give other type of representation formula in \eqref{sol-U_xi}.  Taking the scalar product on both sides of \eqref{t1h} with eigenfunction of $\varphi_{k,\ell}$, and by integration by parts we can see that the Fourier coefficient $\langle U_\xi,\varphi_{k,\ell}\rangle_{L^2(\Omega)}$ in \eqref{sol-U_xi} also reads
\[
\langle U_\xi,\varphi_{k,\ell}\rangle_{L^2(\Omega)} 
=-\frac1{\xi + \lambda_k} \big\langle \,\chi h, D \frac{\partial \varphi_{k,\ell}} {\partial\nu}\big \rangle_{L^2(\partial\Omega)} ,
\]
which further implies the representation of $U_\xi$ as follows
\[
U_\xi = -\sum_{k=1}^\infty \frac1{\xi + \lambda_k} {\sum_{\ell=1}^{m_k} \big\langle \,\chi h, D \frac{\partial \varphi_{k,\ell}} {\partial\nu}\big \rangle_{L^2(\partial\Omega)} \varphi_{k,\ell}} .
\]
On the other hand, in view of the result in Lemma \ref{lem-Lambda}, we have
\[
D\frac{\partial U_\xi}{\partial\nu} \big|_{\Gamma} 
= \tilde D\frac{\partial \tilde U_\xi}{\partial\nu} \big|_{\Gamma},\quad\xi>0,\ h\in H^{\frac32}(\partial\Omega).
\]
Now from the definition of the functions $\Theta_k$ and $\tilde \Theta_k$, noting the notation of the scalar product $\langle \cdot,\cdot \rangle_{L^2(\Omega)}$, we can rephrase the above identity as follows
$$
\begin{aligned}
\sum_{k=1}^\infty \frac1{\xi + \lambda_k} \int_{\partial\Omega} \Theta_k(x,x') \chi(x') h(x')dx'
= \sum_{k=1}^\infty\frac1{\xi + \tilde \lambda_k} \int_{\partial\Omega} \tilde\Theta_k(x,x') \chi(x') h(x')dx',\quad x\in\Gamma.
\end{aligned}
$$
In view of Lemma \ref{lem-U-analy}, the above series are uniformly convergent for $\xi\in K$, where $K$ is any compact subset of $\mathbb C\setminus\{-\lambda_k,-\tilde\lambda_k|\ k\in\mathbb N\}$, therefore from the unique continuation of analytic functions, it follows that
\begin{equation}\label{eq-Theta}
\begin{aligned}
\sum_{k=1}^\infty \frac1{z + \lambda_k} \int_{\partial\Omega} \Theta_k(x,x') \chi(x') h(x')dx'
= \sum_{k=1}^\infty\frac1{z + \tilde \lambda_k} \int_{\partial\Omega} \tilde\Theta_k(x,x') \chi(x') h(x')dx'
\end{aligned}
\end{equation}
for any $x\in\Gamma$ and $z\in \mathbb C\setminus\{-\lambda_k, -\tilde\lambda_k\}_{k=1}^\infty$. 
Now one can show that $\lambda_k = \tilde\lambda_k$, $k\in\mathbb N$ in a similar way to \cite[Theorem 4.2]{Sakamoto11}. Indeed, if not, we assume there exists $k_0\in\mathbb N$ such that $-\lambda_{k_0} \neq \tilde \lambda_{k_0}$, then we can take a sufficiently small disk which includes $-\lambda_{k_0}$ and does not include $\{-\tilde\lambda_k\}_{k=1}^\infty \cup \{-\lambda_k\}_{k\neq k_0}$. Integrating \eqref{eq-Theta} in this disk, we see that
$$
\begin{aligned}
\int_{\partial\Omega} \Theta_{k_0}(x,x') \chi(x') h(x')dx'
= 0,\quad x\in\Gamma.
\end{aligned}
$$
Since here $h\in H^{\frac32}(\partial\Omega)$ is arbitrary chosen, we see from the above identity that $\chi\Theta_{k_0}$ must be vanished on $\Gamma\times\partial\Omega$, which is impossible. This means that $\lambda_k = \tilde\lambda_k$ for any $k\in \mathbb N$.  Repeating the argument used in the proof of \cite[(2.8)]{Kian20} implies that $\Theta_k = \tilde\Theta_k$ on $\Gamma\times\Gamma_{\rm in}'$. This finishes the proof of the lemma.
\end{proof}
On the basis of the results in Lemma \ref{lem-Theta}, we next show a lemma related to the algebraic multiplicity of the eigenvalue $\lambda_k$ and the corresponding eigenfunctions.
\begin{corollary}\label{lem-orth}
Under the same assumption in Lemma \ref{lem-Theta}, then for any $k\in\mathbb N$, there exist orthonormal sets of eigenfunctions $\{\phi_{k,\ell}\}_{\ell=1}^{m_k}$ and $\{\tilde\phi_{k,\ell}\}_{\ell=1}^{\tilde m_k}$ related to $\lambda_k$ of the elliptic operators $A_e$ and $\tilde A_e$ respectively such that
\begin{equation}\label{t1m}
m_k = \tilde m_k,\quad D \frac{\partial \phi_{k,\ell}}{\partial \nu} = \tilde D \frac{\partial \tilde\phi_{k,\ell}}{\partial \nu} \mbox{ on }\partial\Omega,\quad  k\in\mathbb N,\ \ell=1,\ldots,m_k.
\end{equation}
\end{corollary}
Before giving a proof of the lemma, we introduce an algebraic result in \cite{Canuto01} which takes the following form in our context. 
\begin{lemma}\label{lem-canuto}
Let $M_1,M_2\in\mathbb N$ and consider $f_\ell\in C(\partial\Omega),$ $\ell=1,\ldots,M_1$ and $g_\ell\in C(\partial\Omega),$ $\ell=1,\ldots,M_2$. Assume that
\begin{equation}\label{ll1a}
\sum_{\ell=1}^{M_1}f_\ell( x)f_\ell( x')=\sum_{\ell=1}^{M_2}g_\ell( x)g_\ell( x'),\quad x\in\Gamma,\  x'\in \Gamma',
\end{equation}
and that the restrictions $\{f_\ell|_{\Gamma\cap \Gamma'}\}_{\ell=1}^{M_1}$ and $\{g_\ell|_{\Gamma\cap\Gamma'}\}_{\ell=1}^{M_2}$ are linearly independent respectively, where $\Gamma$ and $\Gamma'$ are open subboundaries of $\partial\Omega$ and such that $\Gamma\cup\Gamma'=\partial\Omega$ and $\Gamma\cap\Gamma'\ne\emptyset$. Then there holds $M_1=M_2$ and there exists an $M_1\times M_1$ orthogonal matrix $ O$ such that
\[
(f_1,\ldots,f_{M_1})^{\mathrm T}( x)= O(g_1,\ldots,g_{M_1})^{\mathrm T}( x),\quad x\in\partial\Omega,
\]
where $(\,\cdot\,)^{\mathrm T}$ stands for the transpose.
\end{lemma}
Now on the basis of the above lemma, by an argument similarly to \cite{Kian20}, we can give the proof of Lemma \ref{lem-orth}.
\begin{proof}[\bf Proof of Lemma \ref{lem-orth}]
For any $k\in\mathbb N$, according to the settings in Lemma \ref{lem-canuto}, we denote $M_1:=m_k$, $M_2:=\tilde m_k$ and define the functions $f_\ell$ and $g_\ell$ as follows
$$
\begin{aligned}
 f_\ell(x) =& D(x)\frac{\partial \varphi_{k,\ell}}{\partial\nu}(x),\ \ell=1,\ldots,M_1,\\
 g_\ell(x) =& \tilde D(x)\frac{\partial \tilde\varphi_{k,\ell}}{\partial\nu}(x),\ \ell=1,\ldots,M_2,
\end{aligned}
\quad x\in\partial\Omega.
$$
It is not difficult to check that the above settings fulfill the assumptions in Lemma \ref{lem-canuto}. Indeed, by noting that $\varphi_{k,\ell}$ and $\tilde\varphi_{k,\ell}$ are the eigenfunctions to the elliptic operators $A_e$ and $\tilde A_e$ respectively, we deduce that $\varphi_{k,\ell}$ and $\tilde\varphi_{k,\ell}$ are actually in the function space $C^1(\overline\Omega)$  according to the regularity theories of elliptic equations, which further gives that the functions $f_\ell$, $\ell=1,\ldots,M_1$ and $g_\ell$, $\ell=1,\ldots,M_2$ lie in $C(\partial\Omega)$. 
Moreover, in view of the identities \eqref{eq-Theta} in Lemma \ref{lem-Theta}, the condition \eqref{ll1a} is fulfilled. On the other hand, on the basis of  the unique continuation principle of the solutions for elliptic equations, we see that $\{f_\ell\}_{\ell=1}^{M_1}$ and $\{g_\ell\}_{\ell=1}^{M_2}$ restricted to the subboundary $\Gamma_{\rm in}' \cap \Gamma$ are linearly independent respectively. Therefore, Lemma \ref{lem-canuto} implies that $m_k = \tilde m_k$ and there exists an $m_k\times m_k$ orthogonal matrix $ O_k$ such that
\[
\left(D \frac{\partial \varphi_{k,1}}{\partial\nu},\ldots, D \frac{\partial \varphi_{k,m_k}}{\partial\nu}\right)^{\mathrm T} = O_k\left(\tilde D \frac{\partial \tilde\varphi_{k,1}}{\partial\nu},\ldots, \tilde D \frac{\partial \tilde \varphi_{k,\tilde m_k}}{\partial\nu}\right)^{\mathrm T} \mbox{ on } \partial\Omega.
\]
Finally,  setting $\phi_{k,\ell} := \varphi_{k,\ell}$, $\ell=1,\ldots,m_k$ and
\[
\left(\tilde \phi_{k,1},\ldots, \tilde \phi_{k,m_k}\right)^{\mathrm T} := O_k\left(\tilde \varphi_{k,1},\ldots, \tilde \varphi_{k,m_k}\right)^{\mathrm T},
\]
we see that $\{\phi_{k,\ell}\}_{\ell=1}^{m_k}$ and $\{\tilde\phi_{k,\ell}\}_{\ell=1}^{m_k}$ are the required sequences in Lemma \ref{lem-orth}. This completes the proof of the lemma. 
\end{proof}
We will complete the proof of the first part of Theorem \ref{uniqueness}, that is, the uniqueness determination of the mixture absorption $\mu_a+\mu_f$. To this end, we refer to  the following inverse spectral result in \cite[Corollaries 1.5--1.7]{Canuto04}.

\begin{lemma}\label{p1}
Under the conditions of Theorem \ref{uniqueness}, assume that $\nabla D = \nabla \tilde D$ on boundary $\partial\Omega$ holds. Then \eqref{eq-Theta} and \eqref{t1m} imply $(D, \mu_a+\mu_f) = (\tilde D, \tilde \mu_a + \tilde \mu_f)$.
\end{lemma}
Applying this lemma, we can show the uniqueness of determining the mixture absorption $\mu_a+\mu_f$.
\subsection{Recovering the background absorption and  fluorescence absorption.}
In the above subsection, we can uniquely determine the diffusion coefficient $D$, the mixture absorption $\mu_{a,f}:=\mu_a+\mu_f$ and $u_e$ by the boundary measurement $D\frac{\partial u_e}{\partial\nu}|_{\Gamma\times(0,T)}$. In this part, by substituting $\mu_f=\mu_{a,f} - \mu_a$, we rewrite the problem \eqref{emission-um} into the following form:
\begin{equation}\label{emission-um’}
\begin{cases}
\begin{aligned}
\partial_t u_m + A_0 u_m &+ \mu_a(x) (u_m + u_e) = \mu_{a,f} u_e, && (x,t)\in\Omega\times (0,T),\\
u_m&=0, && x\in\Omega,\, t=0,\\
u_m&=0, && (x,t)\in\partial\Omega\times(0,T).
\end{aligned}
\end{cases}
\end{equation}
Here the operator $A_0u:= - \text{div}(D\nabla u)$, $u\in H^2(\Omega)$. By letting $U_{e,m}:=u_e+u_m$, and noting that $u_e$ solves the system \eqref{excitation-Ue}, that is $\partial_t u_e + A_0 u_e + \mu_{a,f} u_e =0$, we can see that $U_{e,m}$ reads the system
\begin{equation}\label{emission-uem}
\begin{cases}
\begin{aligned}
\partial_t U_{e,m} + A_0 U_{e,m} + \mu_a(x) U_{e,m} & = 0, && (x,t)\in\Omega\times (0,T),\\
U_{e,m}& = 0, && x\in\Omega,\, t=0,\\
U_{e,m}& = g, && (x,t)\in\partial\Omega\times(0,T).
\end{aligned}
\end{cases}
\end{equation}
Now the second observation $D\frac{\partial u_m}{\partial \nu}$ on $\Gamma\times(0,T)$ becomes 
$$
D\frac{\partial U_{e,m}}{\partial\nu} 
= D\left(\frac{\partial u_e}{\partial \nu} + \frac{\partial u_m}{\partial \nu}\right) \mbox{ on }\Gamma\times(0,T).
$$
We can follow an argument similar to the recovery of the mixture absorption in the above subsection to show that the boundary data $D\frac{\partial U_{e,m}}{\partial\nu}|_{\Gamma\times(0,T)}$ can uniquely determine the background absorption  $\mu_a$ in the system \eqref{emission-uem}. Since the mixture absorption $\mu_{a,f}=\mu_a+\mu_f$ is already determined in the above subsection \ref{subsec-uniqu1}, $\mu_f$ can be calculated by $\mu_f=\mu_{a,f} - \mu_a$.

%% file: num.tex
\section{Numerical inversions.}\label{sec_num}
We will consider the numerical inversions in this section. Since solving the simultaneous inverse problems \eqref{inv-new1}-\eqref{inv-new2} is computationally demanding, we have to make a compromise for the DOT-FDOT problem by selecting the parameters we desire to reconstruct. This work focuses on the determination of both background absorption coefficient $\mu_a$ and fluorescence absorption coefficient $\mu_f$. 

\subsection{The Fr\'echet derivatives of observation operators.}
Recall that $\mu_{a,f}:=\mu_a+\mu_f$. We consider the following initial boundary value problems
\begin{equation}\label{homogeneous-Ue}
\begin{cases}
\begin{aligned}
\left(\partial_t -\Delta +\mu_{a,f}(x)\right)u_e&=0, && (x,t)\in\Omega\times (0,T),\\
u_e&=0, && x\in\Omega, \ t=0,\\
u_e&=g, && (x,t)\in\partial\Omega\times(0,T),
\end{aligned}
\end{cases}
\end{equation}
and
\begin{equation}\label{homogeneous-um}
\begin{cases}
\begin{aligned}
\left(\partial_t-\Delta +\mu_a(x)\right)u_m&=\mu_f(x) u_e, && (x,t)\in\Omega\times (0,T),\\
u_m&=0, && x\in\Omega,\ t=0,\\
u_m&=0, && (x,t)\in\partial\Omega\times(0,T).
\end{aligned}
\end{cases}
\end{equation}
Define $U_{e,m}:=u_e+u_m$. Then it satisfies
\begin{equation}\label{homogeneous-ue-um}
\begin{cases}
\begin{aligned}
\left(\partial_t-\Delta +\mu_a(x)\right)U_{e,m}&=0, && (x,t)\in\Omega\times (0,T),\\
U_{e,m}&=0, && x\in\Omega,\ t=0,\\
U_{e,m}&=g, && (x,t)\in\partial\Omega\times(0,T).
\end{aligned}
\end{cases}
\end{equation}
Introduce the admissible set of unknowns by
\begin{equation}\label{admissible-set}
\mathcal{A}:=\big\{q\in L^2(\Omega): 0\leq q \leq C_+\ \text{a.e.\ in} \ \Omega\big\},
\end{equation}
where $C_+>0$ is a constant. The coming inverse problem is to recover the $\mu_a\in\mathcal{A}$ as well as $\mu_f\in\mathcal{A}$ in \eqref{homogeneous-Ue}-\eqref{homogeneous-um} from the time-dependent sub-boundary data
\begin{equation}\label{constant-measure}
\frac{\partial u_e}{\partial\nu}\big|_{\Gamma\times(0,T)}, \quad  \frac{\partial u_m}{\partial\nu}\big|_{\Gamma\times(0,T)}.
\end{equation}

The above simultaneous inverse problem can be divided into following two steps:
\begin{itemize}
\item  {Step 1}: determine the $\mu_{a,f}$, i.e., $\mu_a+\mu_f$, from the data $\frac{\partial u_e}{\partial\nu}|_{\Gamma\times(0,T)}$;
\item  {Step 2}: separate the $\mu_a$ and $\mu_f$ from the recovery of Step 1 by $\big(\frac{\partial u_e}{\partial\nu}+\frac{\partial u_m}{\partial\nu}\big)\big|_{\Gamma\times(0,T)}$.
\end{itemize}
More precisely, we need to solve the following two operator equations
\begin{equation}\label{operator-G1}
\mathbb{G}_e(\mu_{a,f})=\frac{\partial u_e}{\partial\nu}\big|_{\Gamma\times(0,T)}=:\psi_e,\quad (x,t)\in \Gamma\times(0,T)
\end{equation}
and
\begin{equation}\label{operator-G2}
\mathbb{G}_{e,m}(\mu_a)=\big(\frac{\partial u_e}{\partial\nu}+\frac{\partial u_m}{\partial\nu}\big)\big|_{\Gamma\times(0,T)}=:\psi_{e,m}, \quad (x,t)\in \Gamma\times(0,T),
\end{equation}
where $\mathbb{G}_e: \mu_{a,f}\to \psi_e$ and $\mathbb{G}_{e,m}: \mu_a\to \psi_{e,m}$ are the observation operators, and defined by \eqref{homogeneous-Ue} and \eqref{homogeneous-ue-um} respectively. 

To solve the nonlinear equations \eqref{operator-G1} and \eqref{operator-G2} numerically, the gradient-type iteration and Newton-type iteration are common choices. To use such iterations, the differentiability of the operator ${\mathbb G}_e, {\mathbb G}_{e,m}$ and the computations of its Fr$\acute{\text{e}}$chet derivatives are the crucial points. The differentiability of  ${\mathbb G}_e, {\mathbb G}_{e,m}$ for  $\mu_a, \mu_f\in\mathcal{A}$ can be proven in the same way in \cite{zhang19}. Here we only show the expressions of the Fr\'echet derivatives of ${\mathbb G}_e, {\mathbb G}_{e,m}$ in the following. Given the perturbations of $\mu_{a,f}$ and $\mu_a$ as $\tilde\mu_{a,f}$ and $\tilde\mu_a$ respectively, we can compute the Fr\'echet derivatives of ${\mathbb G}_e, {\mathbb G}_{e,m}$ along the corresponding perturbations as
\begin{equation}\label{w1w2}
\begin{cases}
\begin{aligned}
\mathbb{G}'_e(\mu_{a,f})\tilde\mu_{a,f}&:=\lim_{\epsilon\to 0}\frac{\mathbb{G}_e(\mu_{a,f}+\epsilon \tilde\mu_{a,f})-\mathbb{G}_e(\mu_{a,f})}{\epsilon}= \frac{\partial w_e}{\partial\nu} \big|_{\Gamma\times (0,T)}, \\
 \mathbb{G}'_{e,m}(\mu_a)\tilde\mu_a&:=\lim_{\epsilon\to 0}\frac{\mathbb{G}_{e,m}(\mu_a+\epsilon \tilde\mu_a)-\mathbb{G}_{e,m}(\mu_a)}{\epsilon}= \frac{\partial w_{e,m}}{\partial\nu} \big|_{\Gamma\times (0,T)}, 
\end{aligned}
\end{cases}
\end{equation}
where $w_e$ and $w_{e,m}$ satisfy the initial--boundary value problems
\begin{equation}\label{gradient1}
\begin{cases}
\begin{aligned}
\left(\partial_t - \Delta +\mu_{a,f}(x) \right) w_e&=-\tilde\mu_{a,f}(x) u_e[\mu_{a,f}], && (x,t)\in\Omega\times (0,T),\\
w_e&=0, && x\in\Omega,\ t=0,\\
w_e&=0, && (x,t)\in\partial\Omega\times(0,T),
\end{aligned}
\end{cases}
\end{equation}
and
\begin{equation}\label{gradient2}
\begin{cases}
\begin{aligned}
\left(\partial_t -\Delta +\mu_a(x) \right) w_{e,m}&=-\tilde\mu_a(x) U_{e,m}[\mu_a], && (x,t)\in\Omega\times (0,T),\\
w_{e,m}&=0, && x\in\Omega,\ t=0,\\
w_{e,m}&=0, && (x,t)\in\partial\Omega\times(0,T).
\end{aligned}
\end{cases}
\end{equation}

For the exact data $\psi_e$ and $\psi_m$, the equations \eqref{operator-G1} and \eqref{operator-G2} admit unique solutions from Theorem \ref{uniqueness}. However, in practice, we can only have the noisy data $\psi_e^{\delta_e}$ and $\psi_{e,m}^{\delta_{e,m}}$ satisfying ${\|\psi_e^{\delta_e}-\psi_e\|_2}/{\|\psi_e\|_2}\le \delta_e$ and ${\|\psi_{e,m}^{\delta_{e,m}}-\psi_m\|_2}/{\|\psi_m\|_2}\le \delta_{e,m}$ respectively, where $\delta_e, \, \delta_{e,m}\ge 0$ are the noise levels. Then, solving \eqref{operator-G1} and \eqref{operator-G2} are equivalent to solve the operator equation of the form
\begin{equation}\label{uniform-IP}
\mathbb{G}(q)=\psi^\delta,
\end{equation}
where $\mathbb{G}=\mathbb{G}_e, q=\mu_{a,f}, \psi^\delta=\psi_e^{\delta_e}$ for \eqref{operator-G1} and $\mathbb{G}=\mathbb{G}_{e,m}, q=\mu_a, \psi^\delta=\psi_{e,m}^{\delta_{e,m}}$ for \eqref{operator-G2}. Due to the ill-posedness of these inverse problems, we will use the regularized inversion algorithms. 

The gradient-type iteration such as Landweber iteration \cite{Hanke_1995} and the Newton-type iteration such as Levenberg–Maquardt method  \cite{hanke1997regularizing}, the iteratively regularized Gauss-Newton method \cite{JinQ_2013,Kaltenbacher_2015} are common regularization methods to solve \eqref{uniform-IP}.  For instance, the nonlinear Landweber iteration in Hilbert space is defined via
\begin{equation}
q_{n+1}^\delta =q_n^\delta +\big(\mathbb{G}'(q_n^\delta)\big)^*(\psi^\delta-\mathbb{G}(q_n^\delta)),
\end{equation}
where $(\mathbb{G}')^*$ means the conjugate operator of $\mathbb{G}'$, the Levenberg–Maquardt method is defined via
\begin{equation}\label{LM-scheme}
q_{n+1}^\delta =q_n^\delta+\big[(\mathbb{G}'(q_n^\delta))^*\mathbb{G}'(q_n^\delta)+\alpha_n I\big]^{-1}\big(\mathbb{G}'(q_n^\delta)\big)^*(\psi^\delta-\mathbb{G}(q_n^\delta)),
\end{equation}
and the iteratively regularized Gauss–Newton method is defined via
\begin{equation}\label{GN-scheme}
q_{n+1}^\delta =q_n^\delta+\big[(\mathbb{G}'(q_n^\delta))^*\mathbb{G}'(q_n^\delta)+\alpha_n I\big]^{-1} \big(\big(\mathbb{G}'(q_n^\delta)\big)^*(\psi^\delta-\mathbb{G}(q_n^\delta))+\alpha_n(q_0-q_n^\delta)\big).
\end{equation}

The Landweber iteration is quite slow. As an alternative to Landweber-type methods, the Levenberg–Maquardt method or the iteratively regularized Gauss–Newton method requires less number of iterations to satisfy the stopping rule than the Landweber iteration, however each update step of those iterations might take considerably longer than one step of Landweber iteration. In fact, the operator $[(\mathbb{G}'(q_n^\delta))^*\mathbb{G}'(q_n^\delta)+\alpha_n I]^{-1}$ in \eqref{LM-scheme} or \eqref{GN-scheme} has to be solved in practical applications,  implying a huge linear system of equations has to be solved, which often proves to be costly, if not infeasible. Therefore, it may be more desirable to accelerate Landweber iteration by preserving its simple implementation feature. We will introduce an accelerated Landweber iteration scheme in the next subsection.  

\subsection{Nesterov acceleration of Landweber iteration.}
Suppose that $\Omega$ contains a finite number of inclusions (fluorophores) $\{E_j\}_{j=1}^m$, and denote the collection of inclusions by $E=\cup_{j=1}^m E_j$, and assume that there exist constants $d_1, d_2>0$ such that 
\begin{equation*}
{\rm dist}(E,\partial\Omega)>d_1, \; {\rm dist}(E_i,E_j)>d_2\ \text{for}\   1\leq i\ne j\leq m.
\end{equation*}
The unknown $\mu_f$ is a non-smooth function supported on $E$. Furthermore, suppose each inclusion $E_j$ has the fluorescence intensity $\mu_j$. Then we have the fluorescence absorption of the form
\begin{eqnarray*}
\mu_f=
\begin{cases}
\mu_j, &x\in E_j,\; j=1,2,\cdots,m,\\
0, &x\in\Omega\setminus\overline{E}.
\end{cases}
\end{eqnarray*}
The boundary of $E_j$ denotes the shape of $j$-th fluorophore (object) inside the medium. Therefore, when to identify $\mu_{a,f}$,  $\mu_a$ and $\mu_f$ by solving the nonlinear ill-posed equations \eqref{operator-G1} and \eqref{operator-G2}, the interfaces of $\mu_f$ should be kept. 

However, the classical Landweber iteration in Hilbert spaces and the classical $L^2$ penalty term in Tikhonov regularization have the tendency to over-smooth solutions, which makes it difficult to capture special features of the sought solution such as sparsity and discontinuity \cite{Zhong2019}. To overcome this drawback, various regularization schemes in Banach spaces or in a manner of incorporating general non-smooth convex penalty terms have been proposed \cite{Bot2012,Jin_2013,Kaltenbacher_2009}. Further, a so called two-point gradient (TPG) method was proposed in \cite{Hubmer_2017}, which can be considered as some generalization of an accelerated version of Landweber iteration based on Nesterov’s strategy. Some convergence results of this method were also concerned in \cite{Hubmer_2017,Hubmer_2018,Zhong2019}.  Due to the usefulness and acceleration of TPG method in dealing with nonlinear ill-posed problems, we will apply this method to solve our inverse problems.

Now we give a sketch for solving \eqref{uniform-IP} by TPG method for the completeness of this article, which essentially consists of considering
\begin{equation*}
\min_{q\in \mathcal{A}} \mathcal{J}(q), \quad \min_{q\in \mathcal{A}} \mathcal{R}(q)
\end{equation*}
alternatively in the iteration process, where the date-match functional $\mathcal{J}(q)$ is given by 
\begin{equation}
\mathcal{J}(q):=\frac{1}{2}\|\mathbb{G}(q)-\psi^\delta\|_2^2,
\end{equation}
and $\mathcal{R}$ is the $L^2$ combined with TV penalty term given by
\begin{equation}\label{penalty-G1}
\mathcal{R}(q)=\beta_1\|q\|_2^2 +\beta_2|q|_{\rm TV}.
\end{equation}
Pick $q_{-1}=q_0\in \mathcal{A}$ and $\xi_{-1}=\xi_0\in \partial \mathcal{R}(q_0)$ as the initial guesses, where $\partial\mathcal{R}(q):=\big\{\xi\in L^2(\Omega): \mathcal{R}(\tilde q)-\mathcal{R}(q)-\langle \xi,\tilde q-q\rangle \ge 0 \ \text{for all} \ \tilde q\in L^2(\Omega)\big\}$. For $n\ge 0$, we define 
\begin{equation}\label{algorithm}
\begin{cases}
\begin{aligned}
\eta_n&=\xi_n+\lambda_n(\xi_n-\xi_{n-1}),\\
z_n&=\arg \min_{z\in L^2(\Omega)} \big\{\mathcal{R}(z)-\langle \eta_n,z \rangle\big\},\\
\xi_{n+1}&=\xi_n+\alpha_n \big(\mathbb{G}'(z_n)\big)^*(\psi^\delta-\mathbb{G}(z_n)),\\
q_{n+1}&=\arg \min_{q\in L^2(\Omega)} \big\{\mathcal{R}(q)-\langle \xi_{n+1},q \rangle\big\},\\
\end{aligned}
\end{cases}
\end{equation}
where $\lambda_n\ge 0$ is the combination parameter, $\alpha_n$ is the step sizes defined by
\begin{equation*}
\alpha_n=
\begin{cases}
\min\left\{\frac{\bar\alpha_0\|\mathbb{G}(z_n)-\psi^\delta\|_2^2}{\|  (\mathbb{G}'(z_n))^*(\mathbb{G}(z_n)-\psi^\delta) \|_2^2}, \bar\alpha_1\right\} &\text{if} \ \|\mathbb{G}(z_n)-\psi^\delta\|_2>\tau\delta,\\
0 &\text{if} \ \|\mathbb{G}(z_n)-\psi^\delta\|_2 \leq\tau\delta,
\end{cases}
\end{equation*}
for some positive constants $\bar\alpha_0$ and $\bar\alpha_1$.  The combination parameter $\lambda_n\ge 0$ satisfies $\lambda_0=0, \ \lambda_n\in [0,1]$ for all $n\in \mathbb{N}$. Further, the iteration \eqref{algorithm} will be stopped by the discrepancy principle with respect to $z_n^\delta$ i.e., for a given $\tau>1$, we will terminate the iteration after $n_*$ steps, where $n_*:=n(\delta,\psi^\delta)$ is the integer such that
\begin{equation}\label{stop-rule}
\|\mathbb{G}(z_{n_*})-\psi^\delta\|_2 \leq \tau\delta <\|\mathbb{G}(z_{n_*-1})-\psi^\delta\|_2, \quad 0\leq n<n_*.
\end{equation}
Unless otherwise specified, we let the combination parameter $\lambda_n$ be chosen by $\lambda_n=\frac{n}{n+10}$ in this paper,  which leads to a refined version of the Nesterov acceleration of Landweber iteration \cite{Zhong2019}.
\begin{remark}\label{TPG-convergence}
Let $q^\dag\in\mathcal{A}$  be the unique solution of \eqref{uniform-IP} with $\delta=0$. Let the initial guesses $q_0\in\mathcal{A}$ and $\xi_0\in\partial\mathcal{R}(q_0)$ satisfy $\mathcal{R}(q^\dag)-\mathcal{R}(q_0)- \langle \xi_0,q^\dag-q_0 \rangle \leq c_0\rho^2$ for some $\rho>0$. Moreover, let $n_*= n_*(\delta, \psi^\delta)$ be chosen according to the stopping rule \eqref{stop-rule} with $\tau>1$. If the combination parameter $\lambda_n^\delta$ in \eqref{algorithm} is chosen by the discrete backtracking search algorithm \cite{Zhong2019} and the observation operators $\mathbb{G}$ satisfy the tangential condition
\begin{equation}\label{tangential}
\| \mathbb{G}(q)-\mathbb{G}(\tilde q)-\mathbb{G}'(\tilde q)(q-\tilde q)\|_2 \leq \eta \| \mathbb{G}(q)-\mathbb{G}(\tilde q) \|_2, \ \ \text{for all} \  \ q,\tilde{q}\in\mathcal{A},
\end{equation}
there holds $\|q_{n_*}^\delta-q^\dag\|_2\to 0$ as $\delta\to 0$, where $q_{n_*}^\delta$ is the result of \eqref{algorithm} at the $n_*$-th iteration.
\end{remark}

\subsection{Numerical examples.}
We will provide several examples for the implementation of the iterative algorithm.  To run the iteration scheme \eqref{algorithm} for solving our inverse problems, we need to derive $\big(\mathbb{G}_e'(\mu_{a,f})\big)^*$ and $\big(\mathbb{G}_{e,m}'(\mu_a)\big)^*$ respectively. By the definitions of $w_e$ and $w_{e,m}$ in \eqref{w1w2}, we have
\begin{equation}\label{G1G2}
\begin{cases}
\begin{aligned}
\big\langle \mathbb{G}_e'(\mu_{a,f})\tilde\mu_{a,f},\gamma \big\rangle_{L^2(\Gamma\times(0,T))}  &= \int_{\Gamma\times(0,T)}\frac{\partial w_e[\tilde\mu_{a,f}]}{\partial\nu}(x,t)\gamma(x,t)\,dxdt,\\
\big\langle \mathbb{G}_{e,m}'(\mu_a)\tilde\mu_a,\gamma \big\rangle_{L^2(\Gamma\times(0,T))}  &=  \int_{\Gamma\times(0,T)} \frac{\partial w_{e,m}[\tilde\mu_a]}{\partial \nu}(x,t)\gamma(x,t)\,dxdt,
\end{aligned}
\end{cases}
\end{equation}
where $w_e$ and $w_{e,m}$ are the solutions to \eqref{gradient1}-\eqref{gradient2} respectively. Then we can have the representations of $(\mathbb{G}_e'(\mu_{a,f}))^*$ and $(\mathbb{G}_{e,m}'(\mu_f))^*$ by  the identities  
\begin{equation}\label{rela-GG}
\begin{cases}
\begin{aligned}
\big\langle \mathbb{G}_e'(\mu_{a,f})\tilde\mu_{a,f}, \gamma \big\rangle_{L^2(\Gamma\times(0,T))}  &\equiv  \big\langle \tilde\mu_{a,f}, (\mathbb{G}_e'(\mu_{a,f}))^*\gamma \big\rangle_{L^2(\Omega)},\\
\big\langle \mathbb{G}_{e,m}'(\mu_a)\tilde\mu_a, \gamma \big\rangle_{L^2(\Gamma\times(0,T))}  &\equiv  \big\langle \tilde\mu_a, (\mathbb{G}_{e,m}'(\mu_f))^*\gamma \big\rangle_{L^2(\Omega)},
\end{aligned}
\end{cases}
\end{equation}
 for all $\mu_a, \mu_f\in \mathcal{A}$ and $\gamma\in L^2(\Gamma\times(0,T))$. We state the results as the following lemma.
\begin{lemma}
There hold the variational identities given by
\begin{equation}\label{G1-adjoint}
(\mathbb{G}_e'(\mu_{a,f}))^*\gamma=-\int_0^Tu_e[\mu_{a,f}]\phi_e[\gamma] \,dt,
\end{equation}
and
\begin{equation}\label{G2-adjoint}
(\mathbb{G}_{e,m}'(\mu_a))^*\gamma=-\int_{0}^T \big(u_m[\mu_a]+u_e[\mu_{a,f}]\big)\phi_{e,m}[\gamma] \,dt,
\end{equation}
where $\phi_e[\gamma]$ and $\phi_{e,m}[\gamma]$ are the solutions to the initial boundary value problems
\begin{equation}\label{adjoint-gradient1}
\begin{cases}
\begin{aligned}
\left(-\partial_t - \Delta +\mu_{a,f}(x) \right) \phi_e &=0, \qquad\qquad\qquad x\in\Omega, \, t\in(0,T),\\
\phi_e&= 0,  \qquad\qquad\qquad x\in \Omega, \,t=T,\\
\phi_e &=
\begin{cases}
\gamma(x,t),  &  \qquad\; x\in\Gamma, \, t\in (0,T),\\
0, & \qquad\; x\in\partial\Omega\setminus\Gamma, \, t\in (0,T),
\end{cases}
\end{aligned}
\end{cases}
\end{equation}
and
\begin{equation}\label{adjoint-gradient2}
\begin{cases}
\begin{aligned}
\left(-\partial_t - \Delta +\mu_a(x)\right) \phi_{e,m} &=0, \qquad\qquad\qquad x\in\Omega, \, t\in(0,T),\\
\phi_{e,m}&= 0,  \qquad\qquad\qquad x\in \Omega,\, t=T,\\
\phi_{e,m} &=
\begin{cases}
\gamma(x,t),  & \qquad\; x\in\Gamma, \, t\in (0,T),\\
0, & \qquad\; x\in\partial\Omega\setminus\Gamma, \, t\in (0,T).
\end{cases}
\end{aligned}
\end{cases}\end{equation}
\end{lemma}

\begin{proof}
Multiplying two sides of the equation in \eqref{gradient1} by $\phi_e(x,t):=\phi_e[\gamma](x,t)$ and integrating in $\Omega\times(0,T)$, we get
\begin{eqnarray}\label{sun3-5}
\int_{\Omega\times(0,T)} \left(\partial_t - \Delta +\mu_{a,f} \right) w_e \phi_e \,dxdt=
-\int_{\Omega\times(0,T)} \tilde\mu_{a,f} u_e[\mu_{a,f}] \phi_e \,dxdt.
\end{eqnarray}

By $\phi_e(x,T)=0$, $w_e(x,0)=0$ and $\phi_e, w_e\in L^2 (0,T;H^2(\Omega))$, we have
\begin{equation}
\begin{aligned}
\int_{\Omega\times(0,T)}\ \partial_t w_e \,  \phi_e\,  dx dt &={\int_{\Omega} (w_e \phi_e)|_{t=0}^{t=T}\; dx}-\int_{\Omega\times(0,T)} w_e{\partial_t} \phi_e \, dx dt \\
&= -\int_{\Omega\times(0,T)} w_e{\partial_t} \phi_e \, dx dt.
\end{aligned}
\end{equation}
We also have from the Green's formula that
\begin{equation}\label{sun3-7}
\begin{aligned}
\int_{\Omega\times(0,T)}  (- \Delta +\mu_{a,f} )w_e \phi_e\,dxdt
&=\int_{\Omega\times(0,T)}  (- \Delta +\mu_{a,f} )\phi_e w_e \, dxdt \\
&\qquad + \int_0^T\int_{\partial\Omega} \left(\phi_e\frac{\partial w_e}{\partial\nu} -w_e\frac{\partial \phi_e}{\partial\nu}\right) \, dxdt.
\end{aligned}
\end{equation}
Combining \eqref{sun3-5}-\eqref{sun3-7} together and using \eqref{adjoint-gradient1}, we finally have
\begin{equation}\label{sun3-8}
\begin{aligned}
-\int_{\Omega\times(0,T)} \tilde\mu_{a,f} \, u_e[\mu_{a,f}] \,\phi_e[\gamma] \,dxdt
&= \int_0^T \int_{\partial\Omega} \left(\phi_e\frac{\partial w_e}{\partial\nu} -w_e\frac{\partial \phi_e}{\partial\nu}\right) \, dxdt \\
&= \int_0^T \int_{\Gamma} \frac{\partial w_e[\tilde\mu_{a,f}]}{\partial\nu} \gamma dxdt.
\end{aligned}
\end{equation}
Then, \eqref{G1G2} together with \eqref{rela-GG} lead to $(\mathbb{G}_e'(\mu_{a,f}))^*\gamma =\int_0^Tu_e[\mu_{a,f}]\phi_e[\gamma] \,dxdt$. In the same way, we can derive the expression of $(\mathbb{G}_m'(\mu_a))^*\gamma$ which is given as in \eqref{G2-adjoint}. The proof is complete.
\end{proof}

For the noisy data $\psi_{e}^{\delta_e}$ and $\psi_{e,m}^{\delta_{e,m}}$, we describe them as 
$$
\psi_{e}^{\delta_{e} } = \psi_{e}(1+\delta_{e}\zeta)
$$
and
\begin{equation}\label{noisydata}
\psi_{e,m}^{\delta_{e,m} }= \psi_{e,m}(1+\delta_{e,m}\zeta),
\end{equation}
where $\delta_{e}, \, \delta_{e,m}\ge0$ are the noise levels and $\zeta$ is a random standard Gaussian noise. Noting the randomness of the measured noise, we use the mean value of the realizations of the measurements \eqref{noisydata}. Considering the difficulty of obtaining the medical data, the amount of realizations can not be too large and we set it be 10 in our experiments. For the numerical examples, we solve the forward problems \eqref{homogeneous-Ue}-\eqref{homogeneous-ue-um} and dual problems \eqref{adjoint-gradient1}-\eqref{adjoint-gradient2}  by the finite difference method and give an uniform mesh with the time mesh size $0.01$. The primal dual hybrid gradient method will be applied to solve \eqref{algorithm} iteratively. We set $\tau=1.05$ in \eqref{stop-rule} to stop the iteration. In the sequel, denote by $\mu_a^\dag,\,\mu_f^\dag,\,\mu_{a,f}^\dag$ the exact solutions and $\mu_a^N,\,\mu_f^N,\,\mu_{a,f}^N$ the recoveries, where $N$ is the final iteration number. 

%

%
Now we are ready to show the performances of above iterative algorithm for recovering the unknown absorption coefficients. Set $T=1$, $\Omega:=\{x=(x_1,x_2)\in \mathbb{R}^2| \,0\leq x_1\leq1, \,0\leq x_2\leq 1\}$ and suppose $u_e(0,x_2,t)=e^{\frac{x_2}{2}-t}, \, u_e(1,x_2,t)=e^{\frac{1+x_2}{2}-t}, \, u_e(x_1,0,t)=e^{\frac{x_1}{2}-t}$ and $u_e(x_1,1,t)=e^{\frac{1+x_1}{2}-t}$.  We consider the numerical inversions for the following two cases:
\begin{itemize}
\item{(1a)} the homogeneous background absorption and the non-smooth fluoresence absorption given by
\begin{eqnarray}\label{muf-smooth-ex1}
\mu_a^\dag\equiv1, \quad \mu_f^\dag=
\begin{cases}
2, &x\in E,\\
0, &{\rm else},
\end{cases}
\end{eqnarray}
where $E$ is a disc with radius $0.2$, centered at $(0.4,0.4)$;
\item{(1b)} the non-smooth background absorption and the non-smooth fluoresence absorption given by
\begin{eqnarray}\label{muf-nonsmooth-ex1}
\mu_a^\dag =
\begin{cases}
2, &x\in E_1,\\
1, &{\rm else},
\end{cases}
\quad  \mu_f^\dag =
\begin{cases}
2, &x\in E_2,\\
0, &{\rm else},
\end{cases}
\end{eqnarray}
where $E_1,\ E_2$ are two discs with radius $0.15$, centered at $(0.3,0.3)$ and $(0.7,0.7)$, respectively.
\end{itemize}

First, set $\beta_1=0.1, \beta_2=1$ and pick the initial guess as $\xi_0=0$. The results of Step 1 (i.e., recovering $\mu_{a,f}:=\mu_a+\mu_f$) from different noisy level data are listed in Table \ref{ex1tab1} and the recovered solutions are plotted in Figure \ref{fig-ex1a-step1} and Figure \ref{fig-ex1b-step1}, respectively.

Next, fix the noise level of boundary measurement $\frac{\partial u_e}{\partial \nu}|_{\Gamma\times(0,T)}$ as $\delta_e=0.1\%$ for Case (1a) and $\delta_e=0.01\%$ for Case (1b). Still pick the initial guess as $\xi_0=0$, the results of Step 2 from the measurement data with different noisy level $\delta_{e,m}$ are listed in Table \ref{ex1tab2} and the recovered solutions are plotted in Figure \ref{fig-ex1a-step2} and Figure \ref{fig-ex1b-step2}, respectively.

\begin{table}[ht]
\centering
\caption{Numerical results of Step 1.}\label{ex1tab1}
\begin{tabular}{|c|c|c|c|c|}
\hline
$\mu_{a,f}^\dag:=\mu_a^\dag+\mu_f^\dag$  & $\delta_e$       &$\|\mathbb{G}_e(\mu_{a,f}^N)-\psi_e^{\delta_e}\|_2$ &$\frac{\|\mu_{a,f}^N-\mu_{a,f}^\dag\|_2}{\|\mu_{a,f}^\dag\|_2}$ & $N$\\
\hline
Case (1a)      & $1\%$  &  1.5485e-5    & 0.1035  & 605\\
& $0.1\%$   &  1.6742e-7     & 0.0725  & 914\\
& $0$   &  1.1929e-7     & 0.0697 & 6588 \\
\hline
Case (1b)      & $0.1\%$  & 1.3534e-6    & 0.1957  & 2373 \\
& $0.01\%$   &   1.9077e-7    &  0.1110  & 4772\\
& $0$   &   8.4075e-8    &  0.0718  &14980\\
\hline
\end{tabular}
\end{table}

\begin{figure}[ht]
\begin{center}
\includegraphics[width=1\textwidth,height=6.8cm]{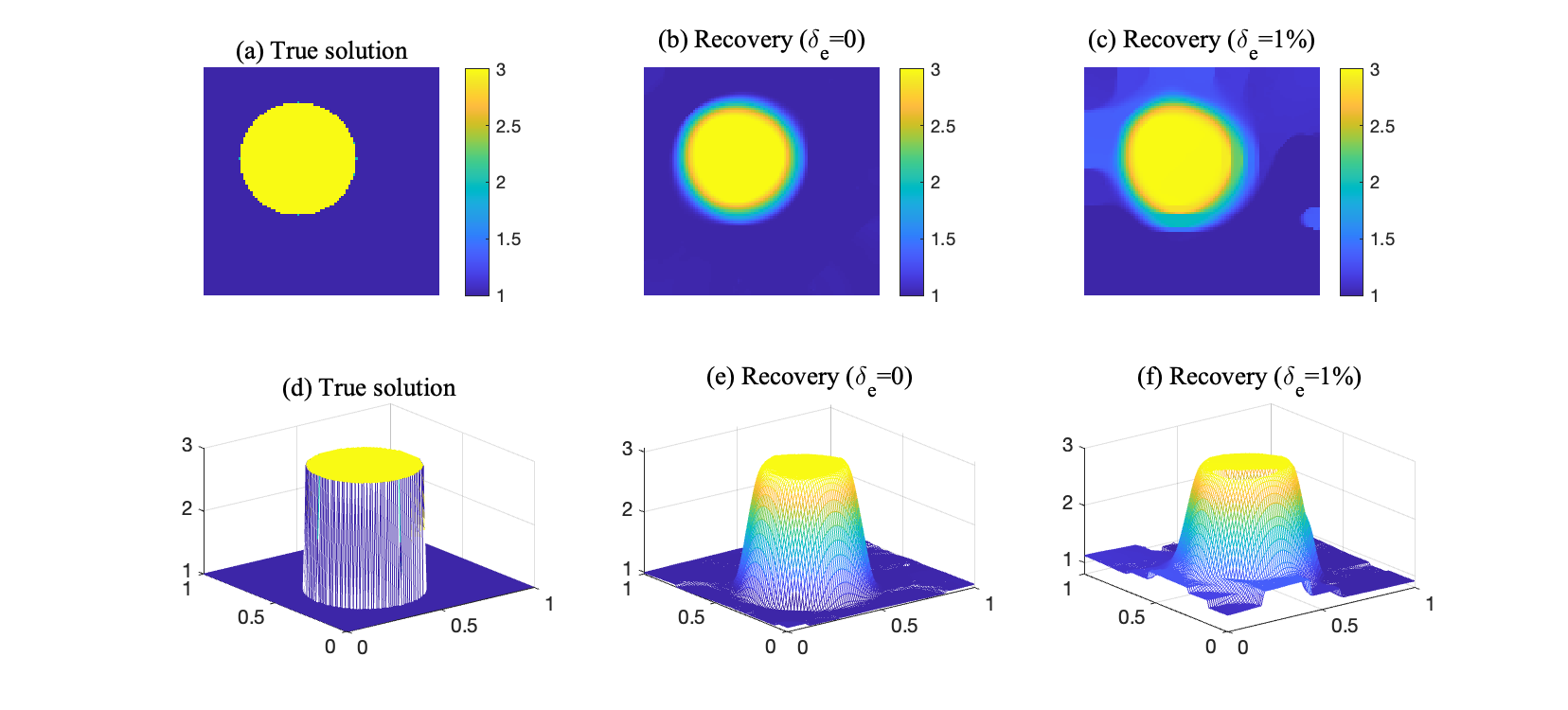}
\caption{The recoveries of $\mu_{a,f}$ in Step 1 for the Case (1a).}\label{fig-ex1a-step1}
\end{center}
\end{figure}

\begin{figure}[ht] 
\begin{center}
\includegraphics[width=1\textwidth,height=6.8cm]{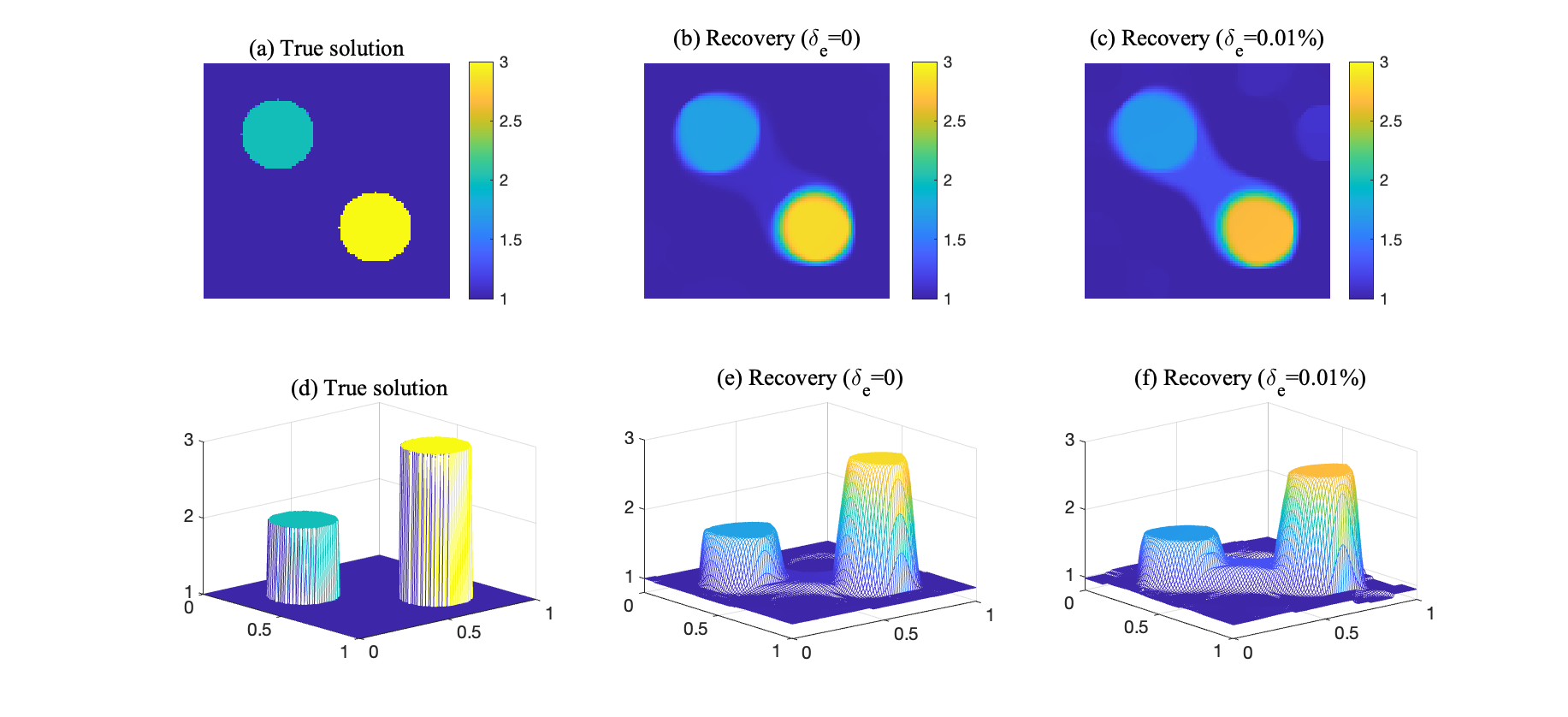}
\caption{The recoveries of $\mu_{a,f}$ in Step 1 for the Case (1b).}\label{fig-ex1b-step1}
\end{center}
\end{figure}

\begin{table}[ht]
\centering
\caption{Numerical results of Step 2 ($\delta_e=0.1\%$).}\label{ex1tab2}
\begin{tabular}{|c|c|c|c|c|c|c|}
\hline
$\mu_a^\dag,\;\mu_f^\dag$  & $\delta_{e}$ & $\delta_{e,m}$      &$\|\mathbb{G}_{e,m}(\mu_f^N)-\psi_{e,m}^{\delta_{e,m}}\|_2$  &$\frac{\|\mu_a^N-\mu_a^\dag\|_2}{\|\mu_a^\dag\|_2}$   &$\frac{\|\mu_f^N-\mu_f^\dag\|_2}{\|\mu_f^\dag\|_2}$ &$N$ \\
\hline
Case (1a)   & $0.1\%$   & $5\%$   &   3.6800e-2   & 0.1315  &0.1931 & 5\\
&  & $1\%$   &   7.4000e-3    &  0.0304 & 0.1344 & 21\\
&  & $0.1\%$   &   7.3593e-4    &  0.0094 & 0.1334 & 129\\
\hline
Case (1b)  & $0.01\%$  & $1\%$  &  1.2439e-5   & 0.1398  & 0.2690 & 597 \\
& & $0.1\%$  &   1.2997e-6 & 0.1111  & 0.2740  & 2262\\
&  & $0.01\%$  &   2.7656e-7  & 0.0794  & 0.2675  & 7950\\
\hline
\end{tabular}
\end{table}

From Tables \ref{ex1tab1}-\ref{ex1tab2} and  Figures \ref{fig-ex1a-step1}-\ref{fig-ex1b-step2}, we see that the recoveries are satisfactory in view of the high ill-posedness of our inverse problems, which is caused by the nonlinearity of this inverse problems and the applied boundary measurements. The algorithm can recover the overall shape of the interface of unknown absorptions. These results indicate that we may recover the piecewise  constant $(\mu_a,\mu_f)$ simultaneously from the boundary data excited by general smooth sources, which is not covered by Theorem \ref{uniqueness} and worth our further investigation. 

However, numerically, due to the non-smoothness, the interfaces of $\mu_{a,f}$, $\mu_a$ or $\mu_f$ are much more challenging to recover and hence less accurately resolved. Especially, as shown in Figure \ref{fig-ex1b-step1}, the recovery of $\mu_{a,f}$ around the part between disc $E_1$ and disc $E_2$ has lower accuracy. Furthermore, as shown in Figure \ref{fig-ex1b-step2}, the recovery error of $\mu_{a,f}$ in Step 1 will influence the recovery accuracy of $\mu_f$ in Step 2, since we need capture the interface of $\mu_f$ by $\mu_{a,f}-\mu_a$. By comparing the results for Case (1a) and Case (1b) as shown in Figures  \ref{fig-ex1a-step2}-\ref{fig-ex1b-step2}, we can see that the non-smooth background absorption and non-smooth fluoresence absorption is more difficult to be accurately separated than the case with homogeneous background absorption. Furthermore,  the numerical results indicate that the algorithm does converge, but it may take many iterations to yield satisfactory recoveries, even if the Nesterov acceleration is applied in the iteration process, reflecting the severe ill-posedness of nonlinear multiple coefficients simultaneous inversion.

\begin{remark}
It is well-known that the convergence and convergence speed of iterative algorithm heavily depend on the initial guess. In above numerical examples, we fixed the initial guess as $\xi_0=0$ and investigated the performance of proposed algorithm. In fact, one can extract the prior information of unknown inclusions $\{E_j\}_{j=1}^m$ by applying our shape-constraint strategy in \cite{Sun19}, which can estimate the positions of unknown inclusions by cuboid approximation. Then we can get a good initial guess for the iterative algorithm in this work and save the number of iteration.
\end{remark}

\begin{figure}[ht] 
\begin{center}
\includegraphics[width=1\textwidth,height=6.8cm]{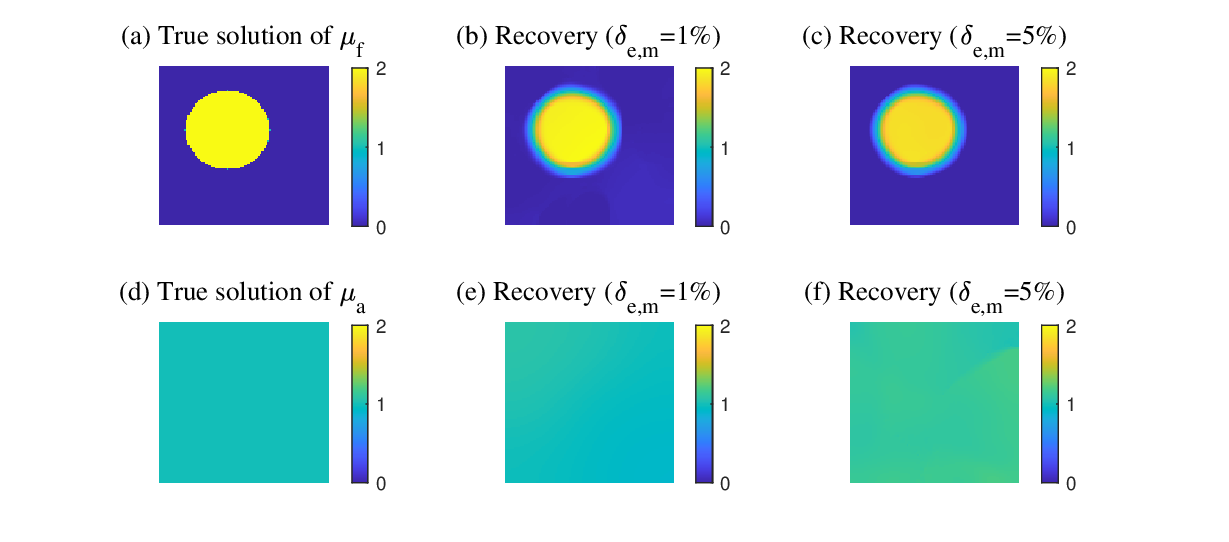}
\caption{The recoveries of $\mu_f$ and $\mu_a$ in Step 2 for the Case (1a).}\label{fig-ex1a-step2}
\end{center}
\end{figure}

\begin{figure}[ht] 
\begin{center}
\includegraphics[width=1\textwidth,height=6.8cm]{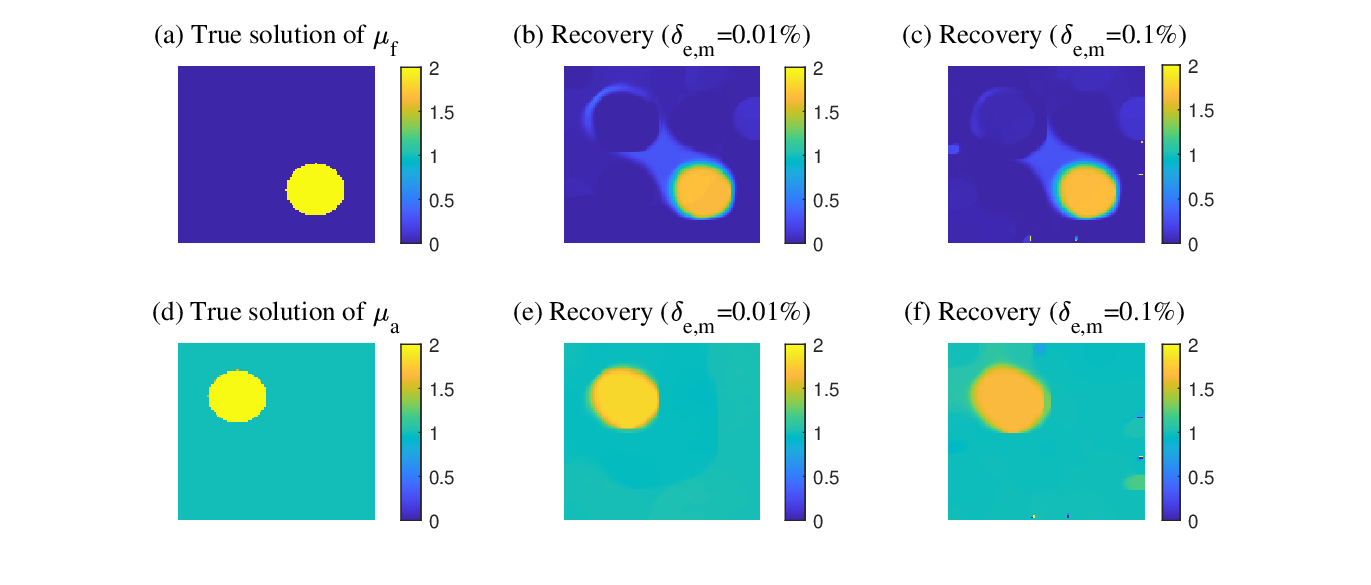}
\caption{The recoveries of Step 2 for the Case (1b).}\label{fig-ex1b-step2}
\end{center}
\end{figure}